\documentclass[10pt]{article}

\usepackage{amsmath,amsthm,amsfonts,amssymb,mathtools}
\usepackage{enumitem}
\usepackage{stmaryrd, tikz-cd}
\usepackage[vmargin=3.5cm]{geometry}

\newtheorem{thm}{Theorem}[section]
\newtheorem{cor}[thm]{Corollary}
\newtheorem{prop}[thm]{Proposition}
\newtheorem{lem}[thm]{Lemma}
\theoremstyle{definition}
\newtheorem{dfn}[thm]{Definition}
\newtheorem{exm}[thm]{Example}
\newtheorem{remark}[thm]{Remark}

\newcommand\tuple[1]{\underline{#1}}

\newcommand\forallst{\forall^{\mathrm{st}}}
\newcommand\forallhyp{\forall^{\mathrm{hyp}}}
\newcommand\existst{\exists^{\mathrm{st}}}
\newcommand\existshyp{\exists^{\mathrm{hyp}}}

\newcommand\restr[2]{{
  \left.\kern-\nulldelimiterspace
  #1
  \vphantom{\big|}
  \right|_{#2}
  }}

\begin{document}
\title{Nonstandard functional interpretations \\ and categorical models \\[5pt] }
\author{Amar Hadzihasanovic\footnote{Department of Computer Science, University of Oxford, Wolfson Building, Parks Road, OX1 3QD Oxford. Email address: amarh@cs.ox.ac.uk. Supported by an EPSRC Doctoral Training Grant.}\, and Benno van den Berg\footnote{Institute for Logic, Language and Computation, Universiteit van Amsterdam, P.O.\ Box 94242, 1090 GE Amsterdam. Email address: bennovdberg@gmail.com. Supported by the Netherlands Organisation for Scientific Research (NWO).}}
\date{4 February 2014}
\maketitle
\vspace{-20pt}
\begin{abstract}
\noindent
Recently, the second author, Briseid and Safarik introduced \emph{nonstandard Dialectica}, a functional interpretation that is capable of eliminating instances of familiar principles of nonstandard arithmetic - including overspill, underspill, and generalisations to higher types - from proofs. We show that, under few metatheoretical assumptions, the properties of this interpretation are mirrored by first order logic in a constructive sheaf model of nonstandard arithmetic due to Moerdijk, later developed by Palmgren. In doing so, we also draw some new connections between nonstandard principles, and principles that are rejected by strict constructivism.

Furthermore, we introduce a variant of the Diller-Nahm interpretion with two different kinds of quantifiers (with and without computational meaning), similar to Hernest's light Dialectica interpretation, and show that one can obtain nonstandard Dialectica from this by weakening the computational content of the existential quantifiers -- a process we call \emph{herbrandisation}. We also define a constructive sheaf model mirroring this new functional interpretation and show that the process of herbrandisation has a clear meaning in terms of these sheaf models.

\end{abstract}
\vspace{-15pt}
\normalsize \tableofcontents

\section{Introduction}
The focus of this paper stands at a confluence of two quite different paths in mathematical logic.

On one end, there is nonstandard arithmetic, and analysis: a subject that has been an upshot of classical model theory, and even after it was recognised that it was amenable to a syntactic treatment, as in Nelson's \emph{internal set theory} \cite{nelson1977internal}, it mostly remained within the boundaries of classical set theory. On the other end, there is the markedly proof-theoretic topic of functional interpretations, stemming from G\"odel's \emph{Dialectica interpretation} \cite{godel1958bisher}; and, in particular, its recent revival through the programme of \emph{proof mining} \cite{kohlenbach2002proof}.

Where these ends meet, is in a general inclination towards the \emph{constructivisation} of mathematics. The first explicit model of nonstandard analysis, due to Schmieden and Laugwitz \cite{schmieden1958erweiterung}, was actually fully constructive, but had a quite weak transfer property. On the other hand, Robinson's model of nonstandard arithmetic \cite{robinson1996nonstandard}, and subsequent ones which were elementary extensions of the standard model, were built from nonconstructive objects, such as nonprincipal ultrafilters of sets. Even in the syntactic approach, it was soon realised that many useful principles led to instances of the excluded middle. But did nonstandard analysis really have \emph{nothing} to offer to constructive analysts?

Not everyone was convinced, including, notably, Per Martin-L\"of, who pushed the question in the early 1990s: first, Erik Palmgren succeeded in building a model with a restricted, yet useful transfer principle \cite{palmgren1995constructive}; then, in 1995, Ieke Moer\-dijk described the first constructive model of nonstandard arithmetic with a \emph{full} transfer principle - a topos of sheaves over a category of filters \cite{moerdijk1995model}. Later, by working in this topos, Palmgren provided simplified, nonstandard proofs of several theorems of constructive analysis, and so demonstrated the usefulness of this model \cite{palmgren1997sheaf, palmgren1998developments, palmgren2000constructive, palmgren2002unifying}.

But if nonstandard proofs do provide some constructive information, we might as well try to extract it in an automated fashion. In 2012, the second author, Briseid and Safarik succeeded in defining a functional interpretation, \emph{nonstandard Dialectica} \cite{van2012functional}, which could eliminate nonstandard principles from proofs of intuitionistic arithmetic in all finite types, enriched, \emph{\`a la} Nelson, with a predicate $\mathrm{st}_\sigma(x)$, ``$x$ is standard'', for all types $\sigma$; also yielding a proof of conservativity of these principles over the base system. Section \ref{sec:nonst_dial} is a review.

Now, some of the principles validated by nonstandard Dialectica were known to hold in Moer\-dijk's topos - including a form of Nelson's idealisation axiom, an underspill principle, and the undecidability of the standardness predicate. Our first aim was to investigate how deep this connection would go.

And a deep connection it is: with the exception of one principle, which requires an assumption about the metatheory, \emph{all} the characteristic principles of nonstandard Dialectica are true in the topos model, for free. Section \ref{sec:filtertopos} is devoted to showing this. During this investigation, we also chanced upon two new principles, \emph{sequence overspill} and \emph{sequence underspill}, which appear to be more natural equivalents of principles that have been taken into consideration, earlier, in the context of proof-theoretic nonstandard arithmetic. We map their relation to other familiar principles from nonstandard and constructive analysis in Section \ref{sec:nonst_dial}.

Several characteristic principles of nonstandard Dialectica have a peculiarity: they are \emph{herbrandised}. This is explained in more detail in Section \ref{sec:uniform}; in short, where ``traditional'' functional interpretations would produce a \emph{single} witness of an existential statement, these principles produce a \emph{finite sequence} of \emph{potential} witnesses, of which at least one is an actual witness. This property destroys the computational meaning of intuitionistic disjunction, yet seems unavoidable in the interpretation of nonstandard arithmetic.

The categorical analysis of nonstandard Dialectica supplied a very convenient way of ``de-herbrandising'', through a simple change in the Grothendieck topology, down from \emph{finite} covers to \emph{singleton} covers. Full transfer is lost - in the new topos, disjunction is stronger than in the metatheory - as well as the link to nonstandard arithmetic; but the de-herbrandised principles induce a new functional interpretation, which we call \emph{uniform Diller-Nahm}, and is the main focus of Section \ref{sec:uniform}.

Uniform Diller-Nahm can be seen as an extension of the Diller-Nahm variant of the Dialectica interpretation \cite{diller1974variante}, and has some striking similarities to light Dialectica \cite{hernest2005light}, a variant of Dialectica with two different kinds of quantifiers - computational, and non computational - introduced in 2005 by Mircea-Dan Hernest, for the purpose of more efficient program extraction from formal proofs. Yet, irrespective of its technical value, the characteristic proof system of uniform Diller-Nahm might have a dignity of its own.

In 1985, Vladimir Lifschitz proposed a simple extension of Heyting arithmetic, where a distinction could be made between calculable, and non calculable natural numbers \cite{lifschitz1985calculable}; a synthesis of classical and intuitionistic arithmetic. Under the interpretation of the predicate $\mathrm{st}(x)$ as ``$x$ is calculable'', the proof system of uniform Diller-Nahm seems to be well-suited for Lifschitz's intended calculus. This is also discussed in Section \ref{sec:uniform}.

Finally, in Section \ref{sec:conclusions}, we survey some open questions.

\emph{Note.} This work is based on research done by the first author, under the supervision of the second author, in partial fulfillment of the requirements for the degree of \emph{Laurea Magistrale} in Mathematics at the University of Pavia.

\section{The nonstandard Dialectica interpretation} \label{sec:nonst_dial}
We start by briefly recalling the definition of the system E-HA$^{\omega*}_\mathrm{st}$, as introduced in \cite{van2012functional}; we refer to the original paper for a detailed presentation.

\subsection{The system $\mathrm{E\mbox{-}HA}^{\omega*}_\mathrm{st}$}
We take E-HA$^{\omega*}$ to be an extension of the system called E-HA$^\omega_0$ in \cite{troelstra1973metamathematical}, with additional types and constants for handling \emph{finite sequences}. More precisely, the collection of types $\textbf{T}^*$ is generated by the inductive clauses
\begin{itemize}
	\item[$\triangleright$] $0$ is in $\textbf{T}^*$;
	\item[$\triangleright$] if $\sigma$, $\tau$ are in $\textbf{T}^*$, then $\sigma \to \tau$ and $\sigma^*$ are in $\textbf{T}^*$;
\end{itemize}
and, for all types $\sigma, \tau$ in $\textbf{T}^*$, we have constants $\langle\rangle_\sigma:\sigma$ (empty sequence), $C: \sigma \to \sigma^* \to \sigma^*$ (prepending operator), and $\mathrm{L}_{\sigma, \tau} : \sigma \to (\sigma \to \tau \to \sigma) \to (\tau^* \to \sigma)$ (list recursor), with defining axioms
\begin{align*}
	& \mathsf{SA}: \quad \forall s: \sigma^* \,(s = \langle \rangle_\sigma \lor \exists x:\sigma\,\exists s': \sigma^* \, (s = Cxs'))\;, \\
	& \begin{matrix*}[l] \begin{cases}
		& \mathrm{L}_{\sigma,\tau} xy\langle\rangle_\tau =_\sigma x\;, \\
		& \mathrm{L}_{\sigma,\tau} xy(Czs) =_\sigma y(\mathrm{L}_{\sigma,\tau} xys)\langle z \rangle\;,
	\end{cases}
		& x : \sigma,\; y: \sigma \to \tau \to \sigma,\; z:\tau, \; s:\tau^*\;,
\end{matrix*}\end{align*}
where $\langle z \rangle$ is the ``singleton'' $Cz\langle\rangle_\tau$.

\begin{itemize} \item[] \textbf{Notation.} We use $s,t,u,v$ (and $s', t',\ldots$) as variables of sequence type.
\end{itemize}
This system has an extensionality axiom
\begin{equation*}
		\forall f,g: \sigma \to \tau \,(f =_{\sigma \to \tau} g \leftrightarrow \forall x:\sigma \, fx =_\tau gx )
\end{equation*}
for all types $\sigma$, $\tau$.

Using the projectors and combinators from the language of E-HA$^\omega_0$, it is possible, already in the latter system, to introduce a coding of finite sequences of elements of any type, as in \cite[p.\ 59]{kohlenbach2008applied}; therefore, E-HA$^{\omega*}$ is a \emph{definitional}, hence conservative, extension of E-HA$^\omega_0$. However, finite sequences seem to be quite ubiquitous in arguments of nonstandard arithmetic, mostly due to the expanded notion of ``finiteness'' in a nonstandard model; so it seems preferable to have them built into our syntax.

Since every type is provably inhabited, we can conservatively add for every type $\sigma$ a constant $\emptyset_\sigma$. Using the list recursor, one can define all the basic operations on finite sequences one needs in practice.
\begin{enumerate}[label=(\roman*)]
	\item A \emph{length function} $|\cdot|:\sigma^* \to 0$, satisfying
	\begin{equation*}
		|\langle\rangle_\sigma| = 0\;, \qquad |Cas| = \mathrm{S}|s|\;,
	\end{equation*}
	for $s:\sigma^*$, $a:\sigma$.
	
	\item A \emph{projection function} $(s,i) \mapsto s_i$ of type $\sigma^* \to 0 \to \sigma$, satisfying
	\begin{align*}
		(\langle\rangle_\sigma)_i & = \emptyset_\sigma \qquad \text{for all } i, \\
		(Cas)_0 & = a\;, \\
		(Cas)_{\mathrm{S}i} & = s_i\;.
	\end{align*}
	
	\item A \emph{concatenation} operation $\cdot: \sigma^* \to \sigma^* \to \sigma^*$, such that
	\begin{equation*}
		\langle\rangle_\sigma \cdot t = t\;, \qquad Cas\cdot t = Ca(s\cdot t)\;.
	\end{equation*}
	As expected, concatenation is provably associative, so we will iterate it without bothering with brackets.		
\end{enumerate}

The following, easy properties are all established in \cite{van2012functional}.
\begin{lem} \label{lem:seq_length}
\begin{enumerate}[label=(\alph*)]
	\item \emph{E-HA}$^{\omega*} \vdash \forall s: \sigma^*\, ( |s| = 0 \leftrightarrow s = \langle \rangle_\sigma )\;,$
	\item \emph{E-HA}$^{\omega*} \vdash \forall n:0 \,\forall s:\sigma^* \,( |s| = \mathrm{S}n \leftrightarrow \exists x:\sigma\,\exists t: \sigma^* \,(s = Cxt \land |t| = n ))\;.$
\end{enumerate}
\end{lem}
\begin{proof} Let $s:\sigma^*$. By the sequence axiom $\mathsf{SA}$, either $s = \langle \rangle_\sigma$ or $s = Cxt$ for some $x:\sigma$, $t:\sigma^*$. If $|s| = 0$, the latter case leads to a contradiction, for $|s| = \mathrm{S}|t| > 0$.

If $|s| = \mathrm{S}n$, then the former case leads to a contradiction, and we have proven the directions left to right. The converses are immediate.
\end{proof}

\begin{prop} \label{prop:ia_star}
\emph{E-HA}$^{\omega*}$ proves the induction schema for sequences
\begin{equation*}
	\mathsf{IA}^*: \quad \big(\varphi(\langle\rangle_\sigma) \land \forall x:\sigma\,\forall s:\sigma^* \,(\varphi(s) \to \varphi(Cxs) )\big) \to \forall s:\sigma^* \, \varphi(s)\;.
\end{equation*}
\end{prop}
\begin{proof}
Suppose $\varphi(\langle\rangle_\sigma)$ and $\forall x:\sigma\,\forall s:\sigma^* \,(\varphi(s) \to \varphi(Cxs) )$. By the previous lemma,
\begin{equation*}
	\forall s:\sigma^* \,(|s| = 0 \to \varphi(s))\;.
\end{equation*}
Fix $n:0$, and assume $\forall s:\sigma^* \,(|s| = n \to \varphi(s))$. Let $s$ be of length $\mathrm{S}n$. Again by the previous lemma, $s = Cxt$ for some $x:\sigma$, and $t: \sigma^*$ of length $n$, and $\varphi(t)$ holds by hypothesis. Therefore, $\varphi(Cxt) \equiv \varphi(s)$ holds as well; and we have proved
\begin{equation*}
	\forall s:\sigma^* \,(|s| = n \to \varphi(s)) \to \forall s:\sigma^* \,(|s| = \mathrm{S}n \to \varphi(s))\;.
\end{equation*}
By ordinary induction, it follows that $\forall n:0, s:\sigma^* \,(|s| = n \to \varphi(s))$.
\end{proof}

\begin{dfn} Let $s, t: \sigma^*$. We say that $s$ and $t$ are \emph{extensionally equal}, and write $s =_e t$, if
\begin{equation*}
	|s| = |t| \land \forall i < |s| \, (s_i = t_i)\;.
\end{equation*}
\end{dfn}
\begin{cor} \label{cor:seq_extensionality}
	\emph{E-HA}$^{\omega*} \vdash \forall s, t: \sigma^*\, ( s =_e t \to s = t )\;.$
\end{cor}
\begin{proof}
By induction for sequences. Suppose $s =_e t$. If $s = \langle \rangle_\sigma$, then $|s| = |t| = 0$, so, by Lemma \ref{lem:seq_length}, $t = \langle \rangle_\sigma$.

Otherwise, $s = Cxs'$ for some $x, s'$. Then $|s| = |t| = \mathrm{S}n$ for $n = |s'|$; again, by Lemma \ref{lem:seq_length}, $t = Cyt'$ for some $y, t'$. But $x = s_0 = t_0 = y$, and $s' =_e t'$; by the inductive hypothesis, $s' = t'$. Therefore, $s = Cxs' = Cyt' = t$.
\end{proof}

Since finite sequences will be used as a replacement for finite sets, we will borrow some set-theoretic notation.
\begin{dfn}
Let $a:\sigma$, $s, s':\sigma^*$. We define the abbreviations
\begin{enumerate}[label=(\roman*)]
	\item $a \in_\sigma s := \exists i<|s| \, (a =_\sigma s_i)$ ($a$ is an \emph{element} of $s$);
	\item $s' \subseteq_\sigma s := \forall x:\sigma \,(x \in_\sigma s' \to x \in_\sigma s)$ ($s'$ is \emph{contained} in $s$).
\end{enumerate}
We will drop subscripts in most occasions. We also extend the relation $\subseteq_\sigma$ to sequence-valued functionals, pointwise: for $s', s: \tau \to \sigma^*$,
\begin{enumerate}
	\item[(iii)] $s' \subseteq s := \forall x: \tau \, (s'x \subseteq_\sigma sx)\;.$
\end{enumerate}
The relation $\subseteq$ determines a preorder, provably in E-HA$^{\omega*}$.
\end{dfn}

In the definition of the nonstandard Dialectica translation, one needs a form of application for finite sequences - and an associated form of $\lambda$-abstraction - that is \emph{monotone} in the first component, with respect to the preorder we just defined.

\begin{dfn}[Finite sequence application and abstraction]
Let $s:(\sigma \to \tau^*)^*$, $a: \sigma$, $t: \tau^*$. Then
\begin{align*}
	s[a] & := (s_0a)\cdot \ldots \cdot(s_{|s|-1}a) : \tau^*\;, \\
	\Lambda x:\sigma.t & := C(\lambda x:\sigma.t)\langle\rangle : (\sigma \to \tau^*)^*\;.
\end{align*}
\end{dfn}
The new application and abstraction are interdefinable with the usual ones. In fact, we have the following, easy compatibility result.
\begin{prop}
\emph{E-HA}$^{\omega*}$ proves that for all $s:\tau^*$, $a: \sigma$,
\begin{equation*}
	(\Lambda x:\sigma.s)[a] = (\lambda x:\sigma.s)a = s[a/x]\;.
\end{equation*}
\end{prop}

\begin{lem} \label{lem:seq_monotonicity}
\emph{E-HA}$^{\omega*}$ proves that for all $s,s':(\sigma \to \tau^*)^*$, $a: \sigma$,
\begin{equation*}
	s \subseteq s' \to s[a] \subseteq s'[a]\;.
\end{equation*}
\end{lem}
\begin{proof}
See \cite[Lemma 2.22]{van2012functional}.
\end{proof}

Since we do not have product types, we will often work with \emph{tuples} of types and of terms, for which we follow the conventions of \cite{kohlenbach2008applied}; the following is a brief summary.
\begin{itemize}
	\item[]\textbf{Notation.} We write $\tuple{\sigma} := \sigma_1,\ldots,\sigma_n$, $\tuple{x}:\tuple{\sigma} := x_0:\sigma_0,\ldots,x_n:\sigma_n$ for tuples of types and terms. $[\,]$ stands for the empty tuple. We write
	\begin{equation*}
		f\tuple{x} := (\ldots(fx_0)x_1)\ldots)x_n\;,
	\end{equation*}
	with the appropriate types; while, if $\tuple{f} := f_0,\ldots,f_m$, $\tuple{f}\tuple{x}$ stands for $f_0\tuple{x},\ldots,f_m\tuple{x}$. We will have, correspondingly,
	\begin{equation*}
		\lambda \tuple{x}.\tuple{f} := \lambda \tuple{x}.f_0, \ldots \lambda \tuple{x}.f_m\;,
	\end{equation*}
	and the same for finite sequence application.
	
	Relations distribute as expected: for instance, if $\tuple{y} := y_0,\ldots,y_n$, with the same length \emph{and} types as $\tuple{x}$,
	\begin{equation*}
		\tuple{x} =_{\tuple{\sigma}} \tuple{y} := \bigwedge_{i=0}^n x_i =_{\sigma_i} y_i\;;
	\end{equation*}
	and if $\tuple{s} := s_0:\sigma_0^*,\ldots,s_n:\sigma_n^*$ is a tuple of sequences,
	\begin{equation*}
		\tuple{x} \in_{\tuple{\sigma}} \tuple{s} := \bigwedge_{i=0}^n x_i \in_{\sigma_i} s_i\;.
	\end{equation*}
\end{itemize}

Most of the results we have listed so far are easily extended to tuples of terms; in particular, those concerning finite sequence application and abstraction.

We now lay the syntactic groundwork for doing nonstandard arithmetic in our system.

\begin{dfn}
The system E-HA$^{\omega*}_\mathrm{st}$ is an extension of E-HA$^{\omega*}$, whose language includes a (unary) predicate $\mathrm{st}_\sigma(x)$, $x:\sigma$, for all types $\sigma$ of $\mathbf{T}^*$; and the \emph{external quantifiers} $\forallst x: \sigma$, $\existst x:\sigma$.

\begin{itemize} \item[] \textbf{Notation.} Following Nelson, so-called \emph{internal} formulae - those in the language of E-HA$^{\omega*}$ - are always denoted with small Greek letters, and generic, \emph{external} formulae with capital Greek letters.
\end{itemize}

The following axioms are added to those of E-HA$^{\omega*}$:
\begin{enumerate}
	\item the defining axioms of the external quantifiers:
	\begin{align*}
		& \forallst x:\sigma \,\Phi(x) \leftrightarrow \forall x:\sigma \,(\mathrm{st}_\sigma(x) \to \Phi(x))\;, \\
		& \existst x:\sigma \,\Phi(x) \leftrightarrow \exists x:\sigma \,(\mathrm{st}_\sigma(x) \land \Phi(x))\;;
	\end{align*}
	
	\item axioms for the standardness predicate:
	\begin{equation*}\begin{matrix*}[l]
		& \mathrm{st}_\sigma(x) \land x =_\sigma y \to \mathrm{st}_\sigma(y) \;, \\
		& \mathrm{st}_\sigma(a) & \text{for all \emph{closed} } a:\sigma\;, \\
		& \mathrm{st}_{\sigma \to \tau}(f) \land \mathrm{st}_\sigma(x) \to \mathrm{st}_\tau(fx)\;;
	\end{matrix*}\end{equation*}
	
	\item the \emph{external induction} schema:
	\begin{equation*}
		\mathsf{IA}^\mathrm{st}: \quad \big(\Phi(0) \land \forallst x:0 \,(\Phi(x) \to \Phi(\mathrm{S}x) )\big) \to \forallst x:0 \, \Phi(x)\;.
	\end{equation*}
\end{enumerate}
Since it is part of E-HA$^{\omega*}$, the system E-HA$^{\omega*}_\mathrm{st}$ also contains, besides the external induction schema, an ``internal'' induction schema $\mathsf{IA}$, which is assumed to hold for internal formulae only.

\end{dfn}
So far, there is nothing inherently nonstandard about the system we have defined. In fact, one could interpret $\mathrm{st}_\sigma(x)$ as $x =_\sigma x$, and all the new axioms would be provable in E-HA$^{\omega*}$. This simple fact also implies that E-HA$^{\omega*}_\mathrm{st}$ is a conservative extension of E-HA$^{\omega*}$.

However, there are some simple results, of the kind we would expect from a ``standardness property'', that can already be proved.

\begin{prop}
For every formula $\Phi(x)$, \emph{E-HA}$^{\omega*}_\mathrm{st}$ proves
\begin{equation*}
	\Phi(x) \land x = y \to \Phi(y)\;.
\end{equation*}
\end{prop}
\begin{proof}
Easy induction on the logical structure of $\Phi$, utilising the fact that the standardness predicate is extensional.
\end{proof}

\begin{prop}
$\text{\emph{E-HA}}^{\omega*}_\mathrm{st} \vdash \forall n, m: 0 \, (\mathrm{st}_0(n) \land m \leq n \to \mathrm{st}_0(m))\;.$
\end{prop}
\begin{proof}
Apply external induction to the formula $\Phi(n) := \forall m:0\, (m \leq n \to \mathrm{st}_0(m))\;$.
\end{proof}

Basically anything one can get from standard sequences is standard.
\begin{lem} \label{lem:st_seq}
\begin{enumerate}[label=(\alph*)]
	\item \emph{E-HA}$^{\omega*}_\mathrm{st} \vdash \forall s: \sigma^*\, ( \mathrm{st}(s) \to \mathrm{st}(|s|) )\;,$
	\item \emph{E-HA}$^{\omega*}_\mathrm{st} \vdash \forall s: \sigma^*\,( \mathrm{st}(s) \to \forall i<|s|\; \mathrm{st}(s_i) )\;,$
	\item \emph{E-HA}$^{\omega*}_\mathrm{st} \vdash \forall s: \sigma^*\, \forall x:\sigma\,( \mathrm{st}(s) \land x \in_\sigma s \to \mathrm{st}(x))\;,$
	\item \emph{E-HA}$^{\omega*}_\mathrm{st} \vdash \forall s,t: \sigma^*\,( \mathrm{st}(s) \land \mathrm{st}(t) \to \mathrm{st}(s\cdot t) )\;,$
	\item \emph{E-HA}$^{\omega*}_\mathrm{st} \vdash \forall f: 0 \to \sigma^*\, \forall n:0\,\big( \mathrm{st}(f) \land \mathrm{st}(n) \to \mathrm{st}(f0\cdot \ldots \cdot fn) \big)\;.$
\end{enumerate}
\end{lem}
\begin{proof}
Everything follows from the standardness axioms, coupled with the fact that the list recursor is standard.
\end{proof}

A simple consequence of the lemma is that the operations of sequence application and abstraction, as defined in the previous section, preserve standardness.
\begin{cor}
\begin{enumerate}[label=(\alph*)]
	\item \emph{E-HA}$^{\omega*}_\mathrm{st} \vdash \forall s: (\sigma \to \tau^*)^* \, \forall x:\sigma\, \big( \mathrm{st}(s) \land \mathrm{st}(x) \to \mathrm{st}(s[x]) \big)\;,$
	\item \emph{E-HA}$^{\omega*}_\mathrm{st} \vdash \forall s: \tau^*\,( \mathrm{st}(s) \to \mathrm{st}(\Lambda x.s ) )\;.$
\end{enumerate}
\end{cor}

Finally, we prove that finite sequences of standard elements are standard; the converse is already a consequence of Lemma \ref{lem:st_seq}.(a)-(b).

\begin{lem} \emph{E-HA}$^{\omega*}_\mathrm{st}$ proves that
\begin{equation*}
	\forall s:\sigma^*\, \big(\mathrm{st}(|s|) \land \forall i < |s| \; \mathrm{st}(s_i) \to \mathrm{st}(s) \big)\;.
\end{equation*}
\end{lem}
\begin{proof}
Suppose $s:\sigma^*$ is finite, and that, for all $i < |s|$, $s_i$ is standard. By an iteration of Lemma \ref{lem:st_seq}.(d), $s' := s_0 \cdot \ldots \cdot s_{|s|-1}$ is also standard. Clearly, $s$ and $s'$ are extensionally equal; by Corollary \ref{cor:seq_extensionality}, $s = s'$. Thus, $s$ is standard.
\end{proof}

This, in turn, is used to prove an external induction schema for sequences.
\begin{prop}
\emph{E-HA}$^{\omega*}_\mathrm{st}$ proves the external induction schema for sequences
\begin{equation*}
	\mathsf{IA}^{*\mathrm{st}}: \quad \big(\Phi(\langle\rangle_\sigma) \land \forallst x:\sigma\,\forallst s:\sigma^* \,(\Phi(s) \to \Phi(Cxs) )\big) \to \forallst s:\sigma^* \, \Phi(s)\;.
\end{equation*}
\end{prop}
\begin{proof}
From the previous lemma, one obtains that if $s = Cxt$ and $s$ is standard, then $x$ and $t$ are also standard. Then one argues precisely as in Proposition \ref{prop:ia_star}, applying external instead of ordinary induction.
\end{proof}

The linguistic blocks are in place for the definition of the nonstandard Dialectica interpretation.

\subsection{The $D_\mathrm{st}$ translation}
\begin{dfn} To every formula $\Phi(\tuple{a})$ of the language of E-HA$^{\omega*}_\mathrm{st}$, with free variables $\tuple{a}$, we associate inductively its \emph{nonstandard Dialectica} translation
\begin{equation*}
	\Phi(\tuple{a})^{D_\mathrm{st}} = \existst \tuple{s}\,\forallst \tuple{y} \, \varphi_{D_\mathrm{st}}(\tuple{s},\tuple{y},\tuple{a})\;,
\end{equation*}
where $\varphi_{D_\mathrm{st}}$ is internal, and all the variables in $\tuple{s}$ are of sequence type.
\begin{itemize}
	\item[$\triangleright$] $\varphi(\tuple{a})^{D_\mathrm{st}} := \varphi_{D_\mathrm{st}}(\tuple{a}) := \varphi(\tuple{a})$, for $\varphi$ internal atomic;
	\item[$\triangleright$] $\mathrm{st}_\sigma(x)^{D_\mathrm{st}} := \existst s:\sigma^* \, ( x \in s )\;$.
\end{itemize}
Let $\Phi(\tuple{a})^{D_\mathrm{st}} = \existst \tuple{s}\,\forallst \tuple{y} \, \varphi_{D_\mathrm{st}}(\tuple{s},\tuple{y},\tuple{a})$, $\Psi(\tuple{b})^{D_\mathrm{st}} = \existst \tuple{t}\,\forallst \tuple{v} \, \psi_{D_\mathrm{st}}(\tuple{t},\tuple{v},\tuple{b})$:
\begin{itemize}
	\item[$\triangleright$] $(\Phi(\tuple{a}) \land \Psi(\tuple{b}))^{D_\mathrm{st}} := \existst \tuple{s},\tuple{t}\,\forallst \tuple{y},\tuple{v}\,\big(\varphi_{D_\mathrm{st}}(\tuple{s},\tuple{y},\tuple{a}) \land \psi_{D_\mathrm{st}}(\tuple{t},\tuple{v},\tuple{b})\big)\;$;
	\item[$\triangleright$] $(\Phi(\tuple{a}) \lor \Psi(\tuple{b}))^{D_\mathrm{st}} := \existst \tuple{s},\tuple{t}\,\forallst \tuple{y},\tuple{v}\,\big(\varphi_{D_\mathrm{st}}(\tuple{s},\tuple{y},\tuple{a}) \lor \psi_{D_\mathrm{st}}(\tuple{t},\tuple{v},\tuple{b})\big)\;$;
	\item[$\triangleright$] $(\Phi(\tuple{a}) \to \Psi(\tuple{b}))^{D_\mathrm{st}} := \existst \tuple{T},\tuple{Y}\,\forallst \tuple{s},\tuple{v}\,\big(\forall \tuple{y} \in \tuple{Y}[\tuple{s},\tuple{v}] \, \varphi_{D_\mathrm{st}}(\tuple{s},\tuple{y},\tuple{a}) \to \psi_{D_\mathrm{st}}(\tuple{T}[\tuple{s}],\tuple{v},\tuple{b})\big)\;$;
	\item[$\triangleright$] $(\exists z \,\Phi(z,\tuple{a}))^{D_\mathrm{st}} := \existst \tuple{s} \, \forallst \tuple{t} \, \exists z \, \forall \tuple{y} \in \tuple{t} \, \varphi_{D_\mathrm{st}}(\tuple{x},\tuple{y},z,\tuple{a})\;$;
	\item[$\triangleright$] $(\forall z \,\Phi(z,\tuple{a}))^{D_\mathrm{st}} := \existst \tuple{s} \, \forallst \tuple{y} \, \forall z \, \varphi_{D_\mathrm{st}}(\tuple{s},\tuple{y},z,\tuple{a})\;$;
	\item[$\triangleright$] $(\existst z \,\Phi(z,\tuple{a}))^{D_\mathrm{st}} := \existst u,\tuple{s} \, \forallst \tuple{t} \, \exists z \in u \, \forall \tuple{y} \in \tuple{t} \, \varphi_{D_\mathrm{st}}(\tuple{s},\tuple{y},z,\tuple{a})\;$;
	\item[$\triangleright$] $(\forallst z \,\Phi(z,\tuple{a}))^{D_\mathrm{st}} := \existst \tuple{S} \, \forallst \tuple{y},z \, \varphi_{D_\mathrm{st}}(\tuple{S}[z],\tuple{y},z,\tuple{a})\;$.
\end{itemize}
\end{dfn}

The idea is that, in the $D_\mathrm{st}$ interpretation, realisers should be finite sequences of \emph{potential} realisers, of which at least one is an actual realiser. Hence, if $s$ is a valid realiser, then any $s'$ with $s \subseteq s'$ should work as well. That this is the case is guaranteed by the following proposition.

\begin{dfn}
A formula $\Phi(s)$ is \emph{upwards closed} in $s:\sigma^*$ if
\begin{equation*}
	\Phi(s) \land s \subseteq s' \to \Phi(s')\;.
\end{equation*}
\end{dfn}

\begin{prop}
Let $\Phi(\tuple{a})$ be a formula of \emph{E-HA}$^{\omega*}_\mathrm{st}$, $\Phi(\tuple{a})^{D_\mathrm{st}} = \existst \tuple{s}\,\forallst \tuple{y} \, \varphi(\tuple{s},\tuple{y},\tuple{a})$. Then $\text{\emph{E-HA}}^{\omega*}$ proves that $\varphi$ is upwards closed in $\tuple{s}$:
\begin{equation*}
	\text{\emph{E-HA}}^{\omega*} \vdash \varphi(\tuple{s},\tuple{y},\tuple{a}) \land \tuple{s} \subseteq \tuple{s}' \to \varphi(\tuple{s}',\tuple{y},\tuple{a})\;.
\end{equation*}
\end{prop}
\begin{proof}
By induction on the logical structure of $\Phi(\tuple{a})$, using Lemma \ref{lem:seq_monotonicity} in the clauses for $\to$ and $\forallst z$.
\end{proof}

In \cite{van2012functional}, the nonstandard Dialectica interpretation was given a characterisation in terms of five principles. We provide here an alternative characterisation, which keeps the following three principles from the former.
\begin{enumerate}
	\item The \emph{herbrandised axiom of choice}:
	\begin{equation*}
		\mathsf{HAC}^\mathrm{st}: \quad \forallst x:\sigma \, \existst y:\tau \, \Phi(x,y) \to \existst (f:\sigma\to\tau^*)^* \, \forallst x:\sigma \, \exists y \in f[x] \, \Phi(x,y)\;.
	\end{equation*}
		
	\item The \emph{herbrandised independence of premise} principle:
	\begin{equation*}
		\mathsf{HIP}_\forallst: \quad (\forallst x:\sigma \, \varphi(x) \to \existst y:\tau \, \Psi(y)) \to \existst t:\tau^* \, (\forallst x:\sigma \, \varphi(x) \to \exists y \in t \, \Psi(y))\;.
	\end{equation*}
	
	\item The principle called \emph{non-classical realisation} in \cite{van2012functional} - which, as we will see in Section \ref{sec:uniform}, could also be called \emph{herbrandised nonstandard uniformity}:
	\begin{equation*}
		\mathsf{NCR}: \quad \forall y: \tau\, \existst x: \sigma \, \Phi(x,y) \to \existst s: \sigma^* \, \forall y: \tau\, \exists x \in s \, \Phi(x,y)\;.
	\end{equation*}
\end{enumerate}
In addition to these, the former characterisation had \emph{idealisation}
\begin{equation*}
	\mathsf{I}: \quad \forallst s:\sigma^* \, \exists y: \tau \, \forall x \in s \, \varphi(x,y) \to \exists y: \tau \, \forallst x : \sigma \, \varphi(x,y)\;,
\end{equation*}	
whose dual
\begin{equation*}
	\mathsf{R}: \quad \forall y: \tau\, \existst x: \sigma \, \varphi(x,y) \to \existst s: \sigma^* \, \forall y: \tau\, \exists x \in s \, \varphi(x,y)\;
\end{equation*}
is clearly a consequence of $\mathsf{NCR}$, and the \emph{herbrandised generalised Markov's principle}
\begin{equation*}
	\mathsf{HGMP}^\mathrm{st}: \quad (\forallst x:\sigma \, \varphi(x) \to \psi) \to \existst s: \sigma^* \, (\forall x \in s \, \varphi(x) \to \psi)\;.
\end{equation*}
We will replace them as follows.
\begin{dfn}
Let $s:\sigma^*$. We say that $s$ is a \emph{hyperfinite enumeration} of the type $\sigma$ if
\begin{equation*}
	\forallst x:\sigma \,(x \in s)\;.
\end{equation*}
We define, for all types $\sigma$, a predicate
\begin{equation*}
	\mathrm{hyper}(s) := \forallst x :\sigma \,(x \in s)\;,
\end{equation*}
as well as quantifiers ranging over hyperfinite enumerations, with defining axioms
\begin{align*}
		& \forallhyp s:\sigma^* \,\Phi(s) := \forall s:\sigma^* \,(\mathrm{hyper}_\sigma(s) \to \Phi(s))\;, \\
		& \existshyp s:\sigma^* \,\Phi(s) := \exists s:\sigma^* \,(\mathrm{hyper}_\sigma(s) \land \Phi(s))\;.
\end{align*}
\end{dfn}
The most basic nonstandard principles are, arguably, overspill and underspill in the type of natural numbers:
\begin{align*}
	\mathsf{OS}_0: & \quad \forallst n:0 \,\varphi(n) \to \exists n:0 \,(\neg\, \mathrm{st}(n) \land \varphi(n))\;, \\
	\mathsf{US}_0: & \quad \forall n:0 \,(\neg\, \mathrm{st}(n) \to \varphi(n)) \to \existst n:0 \,\varphi(n)\;.
\end{align*}
These principles are almost invariably used with formulae of the form $\forall k<n \, \varphi(k)$, stating that a certain property holds \emph{up to} a number $n$. From the assumption $\forallst n \, \forall k<n \, \varphi(k)$, which says that $\varphi$ holds up to \emph{any} standard natural number, $\mathsf{OS}_0$ allows one to derive that $\varphi$ holds up to some nonstandard (infinite) number $n$.

From $n$, one can obtain a hyperfinite enumeration $s := \langle 0, \ldots, n \rangle$ of the natural numbers, so that $\forall k<n \, \varphi(k) \leftrightarrow \forall k \in s \, \varphi(k)$; and, in a way, it is \emph{this} fact - that $n$ induces a hyperfinite enumeration - that is relevant to the argument, rather than $n$ being nonstandard. This suggests the following generalisation of overspill and underspill to all finite types.

We introduce the principle of \emph{sequence overspill}
\begin{equation*}
	\mathsf{OS}^*: \quad \forallst s:\sigma^* \,\varphi(s) \to \existshyp s:\sigma^* \,\varphi(s)\;,
\end{equation*}
and its dual, \emph{sequence underspill}
\begin{equation*}
	\mathsf{US}^*: \quad \forallhyp s:\sigma^* \, \varphi(s) \to \existst s:\sigma^* \,\varphi(s)\;.
\end{equation*}

\begin{prop}\label{prop:ios} \emph{E-HA}$^{\omega*}_\mathrm{st} \vdash \mathsf{I} \leftrightarrow \mathsf{OS}^*\;$.
\end{prop}
\begin{proof}
Assume $\mathsf{I}$, and suppose $\forallst s:\sigma^* \, \varphi(s)$. Let $t:(\sigma^*)^*$ be a standard sequence of sequences; then $s := t_0 \cdot \ldots \cdot t_{|t|-1}$ is again standard, so $\varphi(s)$ holds. Furthermore, by construction, for all $i < |t|$, $t_i \subseteq s$; in other words,
\begin{equation*}
	\forallst t : (\sigma^*)^* \, \exists s : \sigma^* \, \forall t' \in t \,(t' \subseteq s \land \varphi(s))\;.
\end{equation*}
By idealisation, we obtain
\begin{equation*}
	\exists s : \sigma^* \, \forallst t : \sigma^* \,(t \subseteq s \land \varphi(s))\;.
\end{equation*}
It remains to prove that $\forallst t: \sigma^* \, (t \subseteq s) \leftrightarrow \mathrm{hyper}(s)$, an easy consequence of Lemma \ref{lem:st_seq}.

Conversely, assume $\mathsf{OS}^*$, and suppose $\forallst s:\sigma^* \, \exists y: \tau \, \forall x \in s \, \varphi(x,y)$. By sequence overspill, it follows that
\begin{equation*}
	\exists y: \tau \, \existshyp s:\sigma^* \, \forall x \in s \, \varphi(x,y)\;,
\end{equation*}
which implies
\begin{equation*}
	\exists y: \tau \, \forallst x : \sigma^* \, \varphi(x,y)\;.
\end{equation*}
This concludes the proof. \end{proof}

Several consequences of $\mathsf{I}$ are listed in \cite{palmgren1998developments} and in \cite{van2012functional}, which, by the previous proposition, are also consequences of $\mathsf{OS}^*$. For us, it is particularly relevant that $\mathsf{OS}^*$ implies an external version of the \emph{lesser limited principle of omniscience}, a nonconstructive principle well-known in the area of reverse mathematics, see e.g.\ \cite{ishihara2006reverse}; namely,

\begin{equation*}
	\mathsf{LLPO}^\mathrm{st}: \quad \forallst x,y:\sigma \,(\varphi(x) \lor \psi(y)) \to (\forallst x:\sigma \, \varphi(x) \,\lor\, \forallst x:\sigma \, \psi(x) )\;.
\end{equation*}	

\begin{prop} \label{prop:llpost} \emph{E-HA}$^{\omega*}_\mathrm{st} + \mathsf{OS}^* \vdash \mathsf{LLPO}^\mathrm{st} \;$.
\end{prop}
\begin{proof}
Suppose $\forallst x,y:\sigma \,(\varphi(x) \lor \psi(y))$. We prove by external sequence induction that
\begin{equation} \label{eq:pre_overspill}
	\forallst s: \sigma^* \, (\forall x \in s \, \varphi(x) \lor \forall x \in s \, \psi(x))\;.
\end{equation}
For $s = \langle \rangle_\sigma$, $\forall x \in s \, \varphi(x) \lor \forall x \in s \, \psi(x)$ is vacuously true. Suppose it is true for some arbitrary, standard $s$, and pick any standard $a:\sigma$. We want to show $\forall x \in Cas \, \varphi(x) \lor \forall x \in Cas \, \psi(x)$.

Suppose $\forall x \in s \, \varphi(x)$ (the case where $\forall x \in s \, \psi(x)$ is true is similar). Since
\begin{equation*}
	\forallst x,y:\sigma \,(\varphi(x) \lor \psi(y))\;,
\end{equation*}
we have
\begin{equation*}
	\forall b \in Cas \,(\varphi(a) \lor \psi(b))\; ;
\end{equation*}
since $Cas$ is a finite sequence, we can run through all $b \in Cas$ and see whether $\varphi(a)$ holds. If so, then $\forall x \in Cas \, \varphi(x)$ holds and we are done; otherwise, we will get that $\psi(b)$ holds for all $b \in Cas$ and we again achieve the desired disjunction. Now, applying sequence overspill to (\ref{eq:pre_overspill}) gives
\begin{equation*}
	\existshyp s: \sigma^* \, (\forall x \in s \, \varphi(x) \lor \forall x \in s \, \psi(x))\;,
\end{equation*}
which implies $\mathsf{LLPO}^\mathrm{st}$.
\end{proof}

Notice that $\mathsf{OS}_0$ alone would have sufficed to prove the restriction of $\mathsf{LLPO}^\mathrm{st}$ to type 0.

Since $\mathsf{I}$ is equivalent to $\mathsf{OS}^*$, it would make sense if $\mathsf{R}$ were equivalent to $\mathsf{US}^*$; yet things are not so simple. In fact, only one implication seems to hold.

\begin{prop} \emph{E-HA}$^{\omega*}_\mathrm{st} + \mathsf{US}^* \vdash \mathsf{R} \;$.
\end{prop}
\begin{proof}
Suppose $\forall y: \tau\, \existst x: \sigma \, \varphi(x,y)$. Then
\begin{equation*}
	\forallhyp s: \sigma^* \, \forall y: \tau\, \exists x\in s \, \varphi(x,y)\;,
\end{equation*}
which, by sequence underspill, implies $\existst s: \sigma^* \, \forall y: \tau\, \exists x \in s \, \varphi(x,y)$.
\end{proof}

What is missing, in order to obtain an equivalence, is precisely the last characteristic principle.
\begin{prop}  \label{prop:hgmpst}\emph{E-HA}$^{\omega*}_\mathrm{st} + \mathsf{US}^* \vdash \mathsf{HGMP}^\mathrm{st} \;$.
\end{prop}
\begin{proof}
Suppose $\forallst x:\sigma \, \varphi(x) \to \psi$. Then
\begin{equation*}
	\existshyp s:\sigma^* \, \forall x\in s \, \varphi(x) \to \psi\;,
\end{equation*}
which is intuitionistically equivalent to
\begin{equation*}
	\forallhyp s:\sigma^* \, (\forall x\in s \, \varphi(x) \to \psi)\;.
\end{equation*}
An application of sequence underspill leads to the conclusion.
\end{proof}

We now complete the characterisation of $\mathsf{US}^*$.

\begin{prop} \emph{E-HA}$^{\omega*}_\mathrm{st} + \mathsf{HGMP}^\mathrm{st} + \mathsf{R} \vdash \mathsf{US}^* \;$.
\end{prop}
\begin{proof}
Suppose $\forallhyp s:\sigma^* \, \varphi(s)$; that is,
\begin{equation*}
	\forall s:\sigma^* \, (\forallst x : \sigma \, (x \in s) \to \varphi(s))\;.
\end{equation*}
By the herbrandised generalised Markov's principle, this is equivalent to
\begin{equation*}
	\forall s:\sigma^* \, \existst t: \sigma^* \, (t \subseteq s \to \varphi(s))\;;
\end{equation*}
which, by realisation and intuitionistic logic, implies
\begin{equation*} \label{eq:realisation}
	\existst t: (\sigma^*)^* \, \forall s: \sigma^* \, (\forall t' \in t \,(t' \subseteq s) \to \varphi(s))\;.
\end{equation*}

Take a standard $t: (\sigma^*)^*$ as in (\ref{eq:realisation}), and pick $s := t_0 \cdot \ldots \cdot t_{|t|-1}$. By Lemma \ref{lem:st_seq}, $s$ is standard, and, for all $t' \in t$, $t' \subseteq s$; therefore, it holds that $\varphi(s)$. We thus prove
\begin{equation*}
	\existst s:\sigma^* \, \varphi(s)\;,
\end{equation*}
and the sequence overspill principle.
\end{proof}

Replacing $\psi$ with a contradiction, e.g. $0 =_0 1$, and choosing a \emph{negated} $\varphi(x)$, we see that $\mathsf{HGMP}^\mathrm{st}$ - hence, $\mathsf{US}^*$ as well - implies an external version of \emph{Markov's principle}, another noted principle that is rejected by strict constructivism:
\begin{equation*}
	\mathsf{MP}^\mathrm{st}: \quad \big(\forallst x:\sigma\, (\varphi(x) \lor \neg\,\varphi(x)) \land \neg\,\neg\,\existst x:\sigma \, \varphi(x) \big) \to \existst x:\sigma \, \varphi(x)\;.
\end{equation*}
This is another instance of a principle whose nature appears markedly nonstandard, forcing a nonconstructive mode of reasoning.

\begin{thm}[Soundness of the nonstandard Dialectica interpretation]
Suppose
\begin{equation*}
	\text{\emph{E-HA}}^{\omega*}_{\mathrm{st}} + \mathsf{OS}^* + \mathsf{US}^* + \mathsf{NCR} + \mathsf{HAC}^\mathrm{st} + \mathsf{HIP}_\forallst + \Delta_\mathrm{int} \vdash \Phi(\tuple{a})\;,
\end{equation*}
where $\Delta_\mathrm{int}$ is a set of internal sentences. Let $\Phi(\tuple{a})^{D_\mathrm{st}} = \existst \tuple{s}\,\forallst \tuple{y} \, \varphi_{D_\mathrm{st}}(\tuple{s},\tuple{y},\tuple{a})$. Then from the proof we can extract a tuple of closed terms $\tuple{t}$ such that
\begin{equation*}
	\text{\emph{E-HA}}^{\omega*} + \Delta_\mathrm{int} \vdash \forall \tuple{y} \, \varphi_{D_\mathrm{st}}(\tuple{t}, \tuple{y}, \tuple{a})\;.
\end{equation*}
\end{thm}
\begin{proof}
This is \cite[Theorem 5.5]{van2012functional}, coupled with the fact that $\mathsf{OS}^* + \mathsf{US}^* \leftrightarrow \mathsf{I} + \mathsf{R} + \mathsf{HGMP}^\mathrm{st}$ over E-HA$^{\omega*}_\mathrm{st}$. We provide explicit realisers for the new principles.

The interpretation of $\mathsf{OS}^*$ is
\begin{equation*}
	\existst S \, \forallst s' \,\big(\forall s \in S[s'] \, \varphi(s) \to \exists s\,(s' \subseteq s \land \varphi(s))\big)\;,
\end{equation*}
and we can take $S := \Lambda s'.\langle s' \rangle$.

The interpretation of $\mathsf{US}^*$ is
\begin{equation*}
	\existst T \, \forallst s'' \, \big(\forall s \, \exists s' \in s'' \, (s' \subseteq s \to \varphi(s)) \to \exists t \in T[s''] \, \varphi(t) \big);
\end{equation*}
since $\forall s \, \exists s' \in s'' \, (s' \subseteq s \to \varphi(s))$ implies $\varphi(s''_0 \cdot \ldots \cdot s''_{|s''|-1})$, unless $s''$ is the empty sequence (in which case, the premise is false anyway), we can take
\begin{equation*}
	T := \Lambda s''.(s''_0 \cdot \ldots \cdot s''_{|s''|-1})\;. \qedhere
\end{equation*}\end{proof}

\begin{cor} \label{cor:conserv_n}
The system
\begin{equation*}
	\text{\emph{H}} := \text{\emph{E-HA}}^{\omega*}_{\mathrm{st}} + \mathsf{OS}^* + \mathsf{US}^* + \mathsf{NCR} + \mathsf{HAC}^\mathrm{st} + \mathsf{HIP}_\forallst
\end{equation*}
is a conservative extension of \emph{E-HA}$^{\omega*}$, hence of \emph{E-HA}$^{\omega}$.
\end{cor}
\begin{proof}
Follows from the soundness theorem, noting that internal formulae are $D_\mathrm{st}$-interpreted as themselves.
\end{proof}

\begin{thm}[Characterisation of nonstandard Dialectica]

Let $\Phi$ be a formula in the language of $\text{\emph{E-HA}}^{\omega*}_{\mathrm{st}}$.
\begin{enumerate}[label=(\alph*)]
	\item $\text{\emph{H}} \vdash \Phi \leftrightarrow \Phi^{D_\mathrm{st}}\;$.
	\item If for all formulae $\Psi$ of the language of $\text{\emph{E-HA}}^{\omega*}_{\mathrm{st}}$, with $\Psi^{D_\mathrm{st}} = \existst \tuple{s} \, \forallst \tuple{y} \, \psi(\tuple{s}, \tuple{y})$,
	\begin{equation*}
		\text{\emph{H}} + \Phi \vdash \Psi
	\end{equation*}
	implies that there exist closed terms $\tuple{t}$ such that
	\begin{equation*}
		\text{\emph{E-HA}}^{\omega*} \vdash \forall \tuple{y} \, \psi(\tuple{t},\tuple{y})\;
	\end{equation*}
	holds, then $\text{\emph{H}} \vdash \Phi\,$.
\end{enumerate}
\end{thm}
\begin{proof}
See \cite[Theorem 5.8]{van2012functional}.
\end{proof}

Again, we refer to \cite{van2012functional} for proofs of other consequences of the soundness and characterisation theorems, including the closure of H under the \emph{transfer rules}
\begin{align*}
 	& \mathsf{TR}_\forall: \quad \begin{array}{c} \forallst x:\sigma \, \varphi(x) \\ \hline \forall x:\sigma \, \varphi(x)  \end{array}\;, \\
 	& \mathsf{TR}_\exists: \quad \begin{array}{c} \exists x:\sigma \, \varphi(x) \\ \hline \existst x:\sigma \, \varphi(x)  \end{array}\;.
\end{align*}

In summary, the $D_\mathrm{st}$ interpretation is characterised by two reasonable nonstandard principles, and three principles which share the attribute \emph{herbrandised} - something we will later explain in detail. In the next section, we will show that, under the right interpretation of a first order language, these principles are true in Moerdijk's topos of filters.

\section{The filter topos $\mathcal{N}$} \label{sec:filtertopos}
For this section, we assume some basic knowledge about Grothendieck topoi, what it means to interpret a first order language in a Heyting category, and forcing semantics; \cite[Chapter 4]{van1995basic} and \cite[Chapter 6]{mac1992sheaves} can be used as a reference.

\subsection{The filter construction}
In \cite{blass1977two}, Blass introduced a category of filters of sets and ``continuous'' maps between them; rediscovered by Moerdijk, it was used as the underlying category of a site, whose sheaves provided a model of nonstandard arithmetic.

This category arises from $\mathbf{Set}$ as a special case of a general construction - the \emph{filter construction} - whose properties and functoriality were studied by Butz in \cite{butz2004saturated}. When applied on arbitrary categories with finite limits, it can be considered as a completion of the subobject posets under arbitrary meets. We will briefly discuss the general construction, following Butz, before specialising to the case of $\mathbf{Set}$.

We start by recalling the definition of filter on a $\land$-semilattice, i.e.\ on a poset with all finite meets.
\begin{dfn}
Let $S$ be a $\land$-semilattice. A \emph{filter} on $S$ is an inhabited, upwards closed subset of $S$ that is closed under binary meets.

We say that a filter is \emph{proper} if it does not coincide with $S$; otherwise, it is \emph{non proper}.
\end{dfn}
Following Palmgren, we would rather work with \emph{filter bases}, indexed by a set $I$.
\begin{dfn}
A \emph{filter base} $\mathcal{F}_I$ on $S$ is an inhabited set $\{\mathcal{F}_i\}_{i\in I}$ of elements of $S$, such that, for all $i, j \in I$, there exists $k \in I$ such that $\mathcal{F}_k \leq \mathcal{F}_i \land \mathcal{F}_j$.

A filter base \emph{generates} a filter, as follows: $A$ belongs to the filter if and only if there exists $i \in I$ such that $\mathcal{F}_i \leq A$.
\end{dfn}

Notice that a filter base generates a \emph{non} proper filter if and only if it contains the bottom element.

In every category $\mathbf{C}$ with finite limits, the subobject posets are in fact $\land$-semilattices; it is therefore possible to speak of filters of subobjects. That is sufficient to perform the filter construction.

\begin{dfn}
Let $\mathbf{C}$ be a finitely complete category. The \emph{filter category} $\mathfrak{F}\mathbf{C}$ over $\mathbf{C}$ is described by the following data.
\begin{itemize}
	\item Objects are pairs $(C, \mathcal{F}_I)$, where $C$ is an object of $\mathbf{C}$, and $\mathcal{F}_I$ is an $I$-indexed filter base on $\mathrm{Sub}(C)$.
\end{itemize}
We will usually write $\mathcal{F}$ for $(C, \mathcal{F}_I)$, when the underlying object and indexing set are not relevant, and just call it a \emph{filter}. We say that the $\mathcal{F}_i$, $i\in I$, are the \emph{base objects} of the filter.
\begin{itemize}
	\item Morphisms are ``germs of continuous morphisms''. A continuous morphism $\alpha: (C, \mathcal{F}_I) \to (D, \mathcal{G}_J)$ is a partial morphism
	\begin{equation*}
	\begin{tikzcd}[column sep=3ex, row sep=3ex]
		& \mathcal{F}_i \arrow[tail]{dl} \arrow{dr}{\alpha} & \\
		C & & D
	\end{tikzcd}
	\end{equation*}
	in $\mathbf{C}$, defined on some base object $\mathcal{F}_i$, such that for all $j \in J$, there exists $i' \in I$ such that $ \mathcal{F}_{i'} \leq \alpha^* \mathcal{G}_j$ in $\mathrm{Sub}(C)$.
	
	We declare two such morphisms $\alpha: \mathcal{F}_i \to D$, $\alpha' : \mathcal{F}_j \to D$ \emph{equivalent} if there exists $k \in I$ such that $\mathcal{F}_k \leq \mathcal{F}_i \land \mathcal{F}_j$, and $\restr{\alpha}{\mathcal{F}_k} = \restr{\alpha'}{\mathcal{F}_k}$; that is, the following pullback square commutes:
	\begin{equation*}
	\begin{tikzcd}
	\mathcal{F}_k \arrow[tail]{r} \arrow[tail]{d} & \mathcal{F}_i \arrow{d}{\alpha} \\
	\mathcal{F}_j \arrow{r}{\alpha'} & D\;.
	\end{tikzcd}
	\end{equation*}
\end{itemize}
We have an embedding of $\mathbf{C}$ into $\mathfrak{F}\mathbf{C}$, where an object $C$ of $\mathbf{C}$ is identified with the ``simple'' filter $(C, \{C\})$. We will usually still denote the latter with $C$.
\end{dfn}
We will not be overly pedantic about distinguishing between morphisms and their germs, and will write both in the same style.

\begin{lem}
The category $\mathfrak{F}\mathbf{C}$ is finitely complete.
\end{lem}
\begin{proof}
It is sufficient that $\mathfrak{F}\mathbf{C}$ has a terminal object, binary products and equalisers. We give their construction, and omit the proof of the universal properties.

The terminal object is the filter $(1, \{1\})$. The product of $(C, \mathcal{F}_I)$ and of $(D, \mathcal{G}_J)$ is the filter $(C \times D, (\mathcal{F} \times \mathcal{G})_{I \times J})$, where $(\mathcal{F} \times \mathcal{G})_{(i,j)} := \mathcal{F}_i \times \mathcal{G}_j$, for all $i \in I$, $j \in J$.

The equaliser of two morphisms $\alpha, \beta: (C, \mathcal{F}_I) \to (D, \mathcal{G}_J)$, represented by $\alpha: \mathcal{F}_i \to D$ and $\beta: \mathcal{F}_j \to D$, is the inclusion $(C', (\mathcal{F} \land C')_I) \rightarrowtail (C, \mathcal{F}_I)$, where $C'$ is the equaliser of $\alpha$ and $\beta$ in $\mathbf{C}$, and $(\mathcal{F} \land C')_i := \mathcal{F}_i \land C'$ for all $i \in I$.
\end{proof}

\begin{lem}
A morphism $\alpha: \mathcal{F} \to \mathcal{G}$ of $\mathfrak{F}\mathbf{C}$, defined on a base object $\mathcal{F}_i$, is a monomorphism if and only if there exists a base object $\mathcal{F}_j \leq \mathcal{F}_i$ such that $\restr{\alpha}{\mathcal{F}_j}$ is a monomorphism in $\mathbf{C}$.
\end{lem}
\begin{proof}
See \cite[Lemma 2.2]{butz2004saturated}.
\end{proof}

\begin{prop}
For all filters $\mathcal{F}$ in $\mathfrak{F}\mathbf{C}$, $\mathrm{Sub}(\mathcal{F})$ is a meet-complete semilattice, and, for all $\alpha: \mathcal{F} \to \mathcal{G}$, the change of base functor $\alpha^*$ preserves all meets.
\end{prop}
\begin{proof}
By the previous lemma, if $\alpha : (C, \mathcal{F}_I) \rightarrowtail (D, \mathcal{G}_J)$ is a monomorphism, there is some base object $\mathcal{F}_i$ such that $\restr{\alpha}{\mathcal{F}_i}: \mathcal{F}_i \rightarrowtail D$ is a monomorphism in $\mathbf{C}$. Then $(C, \mathcal{F}_I)$ is isomorphic to the filter $(D, (\mathcal{G} \land \mathcal{F}_i)_J)$. It follows that subobjects of $(D, \mathcal{G}_J)$ are in one-to-one correspondence to objects $(D, \mathcal{G}'_{J'})$, such that the base $\mathcal{G}'_{J'}$ generates a filter larger than $\mathcal{G}_J$.

Given an arbitrary family of subobjects $\big\{\big(D, \mathcal{G}^{(i)}_{J^{(i)}}\big)\big\}_{i\in I}$,  let $\mathcal{H}$ be the filter generated by finite meets of the form
\begin{equation*}
	\mathcal{G}^{(i_1)}_{j^{(i_1)}} \land \ldots \land \mathcal{G}^{(i_n)}_{j^{(i_n)}}\;,
\end{equation*}
for $(i_1, \ldots, i_n)$ an arbitrary finite sequence in $I$, and $j^{(i_k)} \in J^{(i_k)}$, $k = 1, \ldots, n$. Then
\begin{equation*}
	\bigwedge_{i \in I} \, \big(D, \mathcal{G}^{(i)}_{J^{(i)}}\big) \simeq (D, \mathcal{H})\;.
\end{equation*}
That this is preserved by change of base can be easily verified by the explicit construction of pullbacks in $\mathfrak{F}\mathbf{C}$.
\end{proof}

An important feature of the filter construction is that it preserves some of the additional properties that $\mathbf{C}$ may have.

\begin{prop}
Let $\mathbf{C}$ be a finitely complete category.
\begin{enumerate}[label=(\alph*)]
	\item If $\mathbf{C}$ is regular, then $\mathfrak{F}\mathbf{C}$ is also regular.
	\item If $\mathbf{C}$ is coherent, then $\mathfrak{F}\mathbf{C}$ is also coherent.
\end{enumerate}
\end{prop}
\begin{proof}
See \cite[Proposition 3.1]{butz2004saturated} and \cite[Proposition 3.2]{butz2004saturated}, respectively.
\end{proof}
Moreover, if $\mathbf{C}$ has all finite coproducts, then $\mathfrak{F}\mathbf{C}$ has them too. In this case, the initial object of $\mathfrak{F}\mathbf{C}$ is the simple filter $(0, \{0\})$; this is isomorphic to any non proper filter $(C, \mathcal{F}_I)$, where $\mathcal{F}_i = 0$ for some $i \in I$. Given two filters $(C, \mathcal{F}_I)$ and $(D, \mathcal{G}_J)$, their coproduct in $\mathfrak{F}\mathbf{C}$ is the filter $(C + D, (\mathcal{F} + \mathcal{G})_{I \times J})$, where $(\mathcal{F} + \mathcal{G})_{(i,j)} := \mathcal{F}_i + \mathcal{G}_j$, for all $i \in I$, $j \in J$.

It is \emph{not}, however, the case that $\mathfrak{F}\mathbf{C}$ is necessarily a Heyting category, when $\mathbf{C}$ is. But this is not a problem, since we really only need $\mathfrak{F}\mathbf{Set}$ to be a coherent category.

As it happens, coherent categories admit a ``natural'' Grothendieck topology, sometimes called the \emph{precanonical} topology: for all objects $C$ of $\mathbf{C}$, a $K$-cover of $C$ is a \emph{finite} family $\{ f_i : C_i \to C \}_{i=1}^n$, such that the union of the images of the $f_i$ is the whole of $C$.

As shown in \cite[Example C2.1.12.(d)]{johnstone2002sketches}, $K$ is \emph{subcanonical}; that is, representable presheaves, of the form $\mathbf{y}C$, for $C$ an object of $\mathbf{C}$, are $K$-sheaves.

Explicitly, for a filter category $\mathfrak{F}\mathbf{C}$, that $\{ \beta_k : \mathcal{G}_k \to \mathcal{F} \}_{k=1}^n$ is a $K$-cover means that, for all choices of base objects $\mathcal{G}_{k, j_k}$ of $\mathcal{G}_k$, $k = 1, \ldots, n$, there exists a base object $\mathcal{F}_i$ of $\mathcal{F}$ such that
\begin{equation*}
	\mathcal{F}_i \leq \beta_1 \mathcal{G}_{1, j_1} \lor \ldots \lor \mathcal{G}_{n, j_n}\;.
\end{equation*}

\begin{dfn}
We will denote the topos $\mathrm{Sh}(\mathfrak{F}\mathbf{Set}, K)$ by $\mathcal{N}$, for \emph{nonstandard} universe.
\end{dfn}

As for all Grothendieck topoi, the \emph{global sections} functor
\begin{equation*}
	\Gamma : \mathcal{N} \to \textbf{Set}\;,
\end{equation*}
sending a sheaf $F$ to the set $\mathrm{Hom}(1, F)$, has a left adjoint $\Delta : \textbf{Set} \to \mathcal{N}$ - the \emph{constant objects} functor. This can be explicitly characterised as follows: for all sets $S$, at all filters $\mathcal{F}$ of $\mathfrak{F}\mathbf{Set}$,
\begin{equation*}
	(\Delta S)\mathcal{F} = \{\alpha: \mathcal{F} \to S \;|\; \alpha \text{ takes a finite number of values}\}\;.
\end{equation*}
Here, $S$ is identified with the simple filter $(S, \{S\})$. It follows that the Yoneda embedding preserves all coproducts of a finite number of copies of 1, but \emph{not} the natural numbers object.

Let $\mathcal{L}$ be a many sorted first order language, and suppose we have fixed an interpretation of $\mathcal{L}$ in $\mathbf{Set}$. We call formulae of $\mathcal{L}$ \emph{internal}, and denote them with small Greek letters. We also want the types of $\mathcal{L}$ to be closed under the clause
\begin{itemize}
	\item[$\triangleright$] if $S$ is a type, then $S^*$ is a type,
\end{itemize}
where $S^*$ is meant to denote the type of finite sequences of elements of type $S$. We will borrow all the notation from the first section in handling finite sequences.

We will identify types, function and relation symbols of $\mathcal{L}$ with their interpretation in $\mathbf{Set}$, and use the standard double square bracket notation for the derived interpretations that we are now going to define. We will take advantage of this semantic overload, and say, for instance, that the type $S$ is inhabited, or that it is infinite, if its interpretation in $\mathbf{Set}$ is; and also that a formula $\varphi$ is \emph{true}, if its interpretation is true in $\mathbf{Set}$.

Let $\mathcal{L}_\mathrm{st}$ be the extension of $\mathcal{L}$ with a unary predicate symbol $\mathrm{st}_S \subseteq S$ for each type $S$. We denote formulae of $\mathcal{L}_\mathrm{st}$ with \emph{capital} Greek letters. We will use abbreviations
\begin{align*}
		\forallst x:S \,\Phi(x) & := \forall x:S \,(\mathrm{st}_S(x) \to \Phi(x))\;, \\
		\existst x:S \,\Phi(x) & := \exists x:S \,(\mathrm{st}_S(x) \land \Phi(x))\;,
\end{align*}
as well as the defined predicate
\begin{equation*}
	\mathrm{hyper}_S(s) := \forallst x : S \,(x \in s)\;,
\end{equation*}
for $s:S^*$, with the relative quantifiers
\begin{align*}
		\forallhyp s:S^* \,\Phi(s) & := \forall s:S^* \,(\mathrm{hyper}_S(s) \to \Phi(s))\;, \\
		\existshyp s:S^* \,\Phi(s) & := \exists s:S^* \,(\mathrm{hyper}_S(s) \land \Phi(s))\;.
\end{align*}
We will often drop the subscript, and just write $\mathrm{st}(x)$, or $\mathrm{hyper}(s)$.

We define an interpretation of $\mathcal{L}_\mathrm{st}$ in $\mathcal{N}$, as follows ($\mathbf{y}$ denotes the Yoneda embedding):
\begin{enumerate}[label=(\roman*)]
	\item for each type $S$, $\llbracket S \rrbracket := \mathbf{y}S$;
	\item for each constant $c:S$, $\llbracket c \rrbracket := \mathbf{y}c: 1 \to \mathbf{y}S$;
	\item for each function symbol $f: S_1,\ldots,S_n \to S$, $\llbracket f \rrbracket := \mathbf{y}f : \mathbf{y}(S_1 \times \ldots \times S_n) \to \mathbf{y}S$;
	\item for each relation symbol $R \subseteq S_1, \ldots, S_n$ of $\mathcal{L}$, $\llbracket R \rrbracket := \mathbf{y}R \rightarrowtail \mathbf{y}(S_1 \times \ldots \times S_n)$;
	\item for each type $S$, $\llbracket \mathrm{st}_S \rrbracket := \Delta S$.
\end{enumerate}
In particular, $\llbracket \mathrm{st}_\mathbb{N} \rrbracket$ is the natural numbers object in $\mathcal{N}$, and the larger sheaf $\llbracket \mathbb{N} \rrbracket$ is a nonstandard model of arithmetic.

The following, fundamental theorem connects the forcing semantics of internal formulae in $\mathcal{N}$ with truth in the metatheory. It is found as \cite[Theorem 1]{palmgren1998developments}, and is an extension of \cite[Lemma 2.1]{moerdijk1995model}.

\begin{thm} \label{thm:transfer}
Let $\varphi(x)$ be an internal formula, with free variable $x$ of type $S$, and $(C, \mathcal{F}_I)$ a filter. For all $\alpha \in \llbracket S \rrbracket \mathcal{F}$,
\begin{equation*}
	\mathcal{F} \Vdash \varphi(\alpha)
\end{equation*}
if and only if there exists $i \in I$ such that, for all $u \in \mathcal{F}_i$, it holds that $\varphi(\alpha(u))$.
\end{thm}

\begin{cor}[Transfer theorem]
Let $\varphi$ be an internal sentence. Then $\varphi$ is true if and only if $\Vdash \varphi$.
\end{cor}

Theorem \ref{thm:transfer} says everything there is to know about internal formulae; we move on to the semantics of the standardness predicate.

\begin{lem}
Let $\mathcal{F}$ be a filter, $S$ a type of $\mathcal{L}$, and $\alpha \in \llbracket S \rrbracket \mathcal{F}$. Then: $\mathcal{F} \Vdash \mathrm{st}_S(\alpha)$ if and only if there exist a $K$-cover $\{\beta_k: \mathcal{G}_k \to \mathcal{F}\}_{k=1}^n$, and elements $x_1, \ldots, x_n \in S$, such that the diagrams
	\begin{equation*}
	\begin{tikzcd}
	\mathcal{G}_k \arrow{r}{\beta_k} \arrow{d}{!} & \mathcal{F} \arrow{d}{\alpha} \\
	1 \arrow{r}{x_k} & S
	\end{tikzcd}
	\end{equation*}
commute in $\mathfrak{F}\mathbf{Set}$, i.e.\ $\alpha\beta_k = x_k!$, for $k = 1,\ldots, n$.
\end{lem}
\begin{proof} Follows immediately from the interpretation chosen for the standardness predicate, and the description of $\Delta S$.
\end{proof}

\begin{lem} \label{lem:quantifiers}
Let $\Phi(x,y)$ be an external formula, with free variables $x:S$ and $y:T$, $\mathcal{F}$ a filter, and $\alpha \in \llbracket S \rrbracket \mathcal{F}$. Then:
\begin{enumerate}[label=(\alph*)]
	\item $\mathcal{F} \Vdash \forallst y:T \, \Phi(\alpha, y)$ if and only if, for all $y \in T$, $\mathcal{F} \Vdash \Phi(\alpha, y!)$;
	\item $\mathcal{F} \Vdash \existst y:T \, \Phi(\alpha, y)$ if and only if there exist a $K$-cover $\{\beta_k: \mathcal{G}_k \to \mathcal{F}\}_{k=1}^n$, and elements $y_1, \ldots, y_n \in T$, such that
	\begin{equation*}
		\mathcal{G}_k \Vdash \Phi(\alpha\beta_k, y_k!)\;, \qquad k = 1,\ldots,n\;,
	\end{equation*}
	or, equivalently, there exists $t \in T^*$ such that $\mathcal{F} \Vdash \exists y \in t! \, \Phi(\alpha, y)$;
\end{enumerate}
\end{lem}
\begin{proof}
See \cite[Lemma 3.3]{palmgren1997sheaf}.
\end{proof}

With these lemmata, we are able to prove that the simple axioms that we imposed on the standardness predicate in the system E-HA$^{\omega*}_\mathrm{st}$ hold in $\mathcal{N}$. That the predicate respects equality is immediate; that closed terms are standard amounts, in this context, to the fact that, for all types $S$ and elements $\alpha \in \llbracket S \rrbracket 1$,
\begin{equation*}
	\Vdash \mathrm{st}_S(\alpha)\;,
\end{equation*}
as any morphism $\alpha: 1 \to S$ is obviously constant.

\begin{prop}
For all types $S$, $T$, the following statement is true in $\mathcal{N}$:
\begin{equation*}
	\forallst f:S\to T \, \forallst x:S \; \mathrm{st}_T(f(x))\;.
\end{equation*}
\end{prop}
\begin{proof}
By Lemma \ref{lem:quantifiers}, $\Vdash \forallst f:S\to T \, \forallst x:S \; \mathrm{st}_T(f(x))$ if and only if, for all $f \in (S \to T)$, and $x \in S$,
\begin{equation*}
	\Vdash \mathrm{st}_T(f!(x!))\;.
\end{equation*}
But $\llbracket f!(x!) \rrbracket = \llbracket f(x)! \rrbracket$, and the latter is clearly standard.
\end{proof}

\begin{prop}
The external induction schema $\mathsf{IA}^\mathrm{st}$ holds in $\mathcal{N}$.
\end{prop}
\begin{proof}
Let $\Phi(x,n)$ be an external formula, with $x:S$ and $n:\mathbb{N}$, $\mathcal{F}$ a filter, and $\alpha \in \llbracket S \rrbracket \mathcal{F}$. Suppose
\begin{equation*}
	\mathcal{F} \Vdash \Phi(\alpha, 0!) \land \forallst n:\mathbb{N} \, (\Phi(\alpha, n) \to \Phi(\alpha, \mathrm{S}n) )\;.
\end{equation*}
Then, by Lemma \ref{lem:quantifiers}, we have that $\mathcal{F} \Vdash \Phi(\alpha, 0!)$ and that, for all $n \in \mathbb{N}$, $\mathcal{F} \Vdash \Phi(\alpha, n!)$ implies $\mathcal{F} \Vdash \Phi(\alpha, \mathrm{S}n!) $. By induction in the metatheory, we obtain that, for all $n \in \mathbb{N}$,
\begin{equation*}
	\mathcal{F} \Vdash \Phi(\alpha, n!)\;,
\end{equation*}
so, again by the semantics of the external quantifiers, $\mathcal{F} \Vdash \forallst n:\mathbb{N} \, \Phi(\alpha, n)$.
\end{proof}

Lemma \ref{lem:quantifiers} has also the following easy consequence.
\begin{cor} \label{cor:trforallexists}
Let $\varphi$ be an internal formula. Then $\Vdash \forallst x:S \, \existst y:T \, \varphi(x,y)$ if and only if it is true that $\forall x \in S \, \exists y \in T \, \varphi(x,y)$. Equivalently, the rule
\begin{align*}
 	& \mathsf{TR}_{\forall\exists}: \quad \begin{array}{c} \forall x:\sigma \, \exists y:\tau \, \varphi(x,y) \\ \hline \forallst x:\sigma \, \existst y:\tau \, \varphi(x,y)  \end{array}\;.
\end{align*}
holds in $\mathcal{N}$.
\end{cor}

We have, by now, a good picture of the semantics of first order logic in the filter topos $\mathcal{N}$. In the next section, we will deal with the characteristic principles of nonstandard Dialectica.

\subsection{Characteristic principles}
For the results in this section, we cannot take much credit, since a characterisation of first order logic in the topoi $\mathrm{Sh}(\mathfrak{F}\mathbf{C}, K)$, with $\mathbf{C}$ coherent, has already been provided by Butz \cite[Proposition 4.5]{butz2004saturated}, albeit with a different aim and formalism. The choice of principles, however, is different, due to our focus on nonstandard arithmetic; moreover, it will allow us to see \emph{herbrandisation} ``in action'', once we decide to ``de-herbrandise'' in the following section.

We start from the truly nonstandard principles, sequence overspill and underspill.

\begin{prop}
The principle $\mathsf{OS}^*$ holds in $\mathcal{N}$.
\end{prop}
\begin{proof}
Let $\varphi(y,s)$ be an internal formula, with variables $y: T$, and $s: S^*$, $(C, \mathcal{F}_I)$ an arbitrary filter, $\alpha \in \llbracket T \rrbracket \mathcal{F}$, and assume
\begin{equation*}
	\mathcal{F} \Vdash \forallst s:S^* \, \varphi(\alpha, s)\;.
\end{equation*}
By Lemma \ref{lem:quantifiers}, for all $s \in S^*$, $\mathcal{F} \Vdash \varphi(\alpha, s!)$; by transfer (Theorem \ref{thm:transfer}), for all $s \in S^*$, there exists $i \in I$ such that, for all $u \in \mathcal{F}_i$,
\begin{equation*}
	\varphi(\alpha(u), s)\;.
\end{equation*}
Define a filter $(C \times S^*, \mathcal{G}_{I \times S^*})$, as follows: for all $i \in I$, $t \in S^*$,
\begin{equation*}
	\mathcal{G}_{(i, t)} := \{ (u, s) \,|\, u \in \mathcal{F}_i \land t \subseteq s \land \varphi(\alpha(u), s) \}\;.
\end{equation*}
The filter condition is easily checked: given $\mathcal{G}_{(i, t)}$, $\mathcal{G}_{(j, t')}$, pick $k \in I$ such that $\mathcal{F}_k \subseteq \mathcal{F}_i \cap \mathcal{F}_j$, and $t'' := t\cdot t'$; then, $\mathcal{G}_{(k, t'')} \subseteq \mathcal{G}_{(i, t)} \cap \mathcal{G}_{(j, t')}$.

The projections $\pi_1 : \mathcal{G} \to \mathcal{F}$, and $\pi_2 : \mathcal{G} \to S^*$ are clearly continuous. We now check
\begin{equation*}
	\mathcal{G} \Vdash \mathrm{hyper}(\pi_2)\;.
\end{equation*}
By definition, this means $\mathcal{G} \Vdash \forallst x:S \, (x \in \pi_2)$; equivalently, for all $x \in S$, $\mathcal{G} \Vdash x! \in \pi_2$. By transfer, it suffices to prove that, for all $x \in S$, there exists $(i,t) \in I \times S^*$, such that for all $u \in \mathcal{F}_i$, and $s \supseteq t$, it holds that $x \in s$; so we can take $t := \langle x \rangle$, and $i \in I$ arbitrary.

Furthermore, $\mathcal{G} \Vdash \varphi(\alpha\pi_1, \pi_2)$ holds by construction. Hence, in order to derive that
\begin{equation*}
	\mathcal{F} \Vdash \exists s:S^* \, (\mathrm{hyper}(s) \land \varphi(\alpha, s))\;,
\end{equation*}
it remains to show that $\pi_1$ is covering. Let $\mathcal{G}_{(i,t)}$ be an arbitrary base set of $\mathcal{G}$. By the assumption, we can find $j \in I$ such that, for all $u \in \mathcal{F}_j$, $\varphi(\alpha(u), t)$; then, if we choose $k \in I$ such that $\mathcal{F}_k \subseteq \mathcal{F}_i \cap \mathcal{F}_j$, we have that
\begin{equation*}
	\mathcal{F}_k \subseteq \pi_1 \,\mathcal{G}_{(i,t)}\;.
\end{equation*}
This concludes the proof.
\end{proof}

\begin{lem}
Let $\Phi(x)$ be an external formula, $x:S$, such that
\begin{equation} \label{eq:ushyp}
	\Vdash \exists x: S \, \Phi(x)\;.
\end{equation}
Then
\begin{equation*}
	\forall y: T \, \big(\forall x: S \, (\Phi(x) \to \varphi(y, x)) \to \existst x:S \, \varphi(y, x) \big)
\end{equation*}
holds in $\mathcal{N}$ for all internal formulae $\varphi$.
\end{lem}
\begin{proof}
Let $\mathcal{F}$ be any filter, $\varphi(y, x)$ an internal formula, $y:T$, and $\alpha \in \llbracket T \rrbracket \mathcal{F}$. Suppose $\mathcal{F} \Vdash \forall x:S \, (\Phi(x) \to \varphi(\alpha, x))$; equivalently,
\begin{equation} \label{eq:ushyp_2}
	\mathcal{F} \times S \Vdash \Phi(\pi_2) \to \varphi(\alpha\pi_1, \pi_2)\;.
\end{equation}

Assume (\ref{eq:ushyp}). Then, there exist a cover $\{\mathcal{G}_k \to 1\}_{k=1}^n$, and elements $\sigma_k \in \llbracket S \rrbracket\mathcal{G}_k$, $k=1,\ldots,n$, such that
\begin{equation*}
	\mathcal{G}_k \Vdash \Phi(\sigma_k)\;, \quad k = 1,\ldots, n\;.
\end{equation*}
By our interpretation of the type $S$, the $\sigma_k$ correspond to morphisms $\sigma_k: \mathcal{G}_k \to S$ in $\mathfrak{F}\mathbf{Set}$; by monotonicity of the forcing relation, we obtain
\begin{equation*}
	\mathcal{F} \times \mathcal{G}_k \Vdash \Phi(\sigma_k\pi_2)\;,
\end{equation*}
which, by the commutativity of the diagrams
\begin{equation*}
\begin{tikzcd}
	\mathcal{F} \times \mathcal{G}_k \arrow{r}{\pi_2} \arrow{d}{\mathrm{id} \times \sigma_k} & \mathcal{G}_k \arrow{d}{\sigma_k} \\
	\mathcal{F} \times S \arrow{r}{\pi_2} & S \;,
\end{tikzcd}
\end{equation*}
for $k = 1,\ldots,n$, is the same as $\mathcal{F} \times \mathcal{G}_k \Vdash \Phi(\pi_2(\mathrm{id} \times \sigma_k))$.

Therefore, from (\ref{eq:ushyp_2}), it follows, by monotonicity, that
\begin{equation*}
	\mathcal{F} \times \mathcal{G}_k \Vdash \varphi(\alpha\pi_1, \sigma_k\pi_2)\;;
\end{equation*}
by transfer, for all $k = 1, \ldots, n$, there exist base sets $\mathcal{F}_{i_k}$ of $\mathcal{F}$, $\mathcal{G}_{k, j_k}$ of $\mathcal{G}_k$, such that for all $u \in \mathcal{F}_{i_k}$, and $v \in \mathcal{G}_{k,j_k}$, it holds that $\varphi(\alpha(u), \sigma_k(v))$.

Now, since the $\mathcal{G}_k$ cover 1, there exists some $x \in \sigma_1\mathcal{G}_{1,j_1} \cup \ldots \cup \sigma_n\mathcal{G}_{n,j_n}$. For such an $x$, taking $\mathcal{F}_i \subseteq \mathcal{F}_{i_1} \cap \ldots \cap \mathcal{F}_{i_n}$, and using transfer,
\begin{equation*}
	\mathcal{F} \Vdash \varphi(\alpha, x!)\;;
\end{equation*}
hence $\mathcal{F} \Vdash \existst x :S \, \varphi(\alpha, x)$, which was to be proved.
\end{proof}

\begin{prop}
The principle $\mathsf{US}^*$ holds in $\mathcal{N}$.
\end{prop}
\begin{proof} Follows from the previous lemma, by taking $\Phi(s) := \mathrm{hyper}(s)$, and using for condition (\ref{eq:ushyp}) the fact that, by sequence overspill, hyperfinite enumerations of any type exist in $\mathcal{N}$.
\end{proof}

Given sequence overspill and underspill, one can adapt the proofs of the first section to show that other principles, including idealisation and the herbrandised generalised Markov's principle, hold in $\mathcal{N}$. However, one should pay attention to the fact that, while finite types were all inhabited, and actually had infinitely many elements, in this context a type $S$ can be finite, or even empty. So, for instance, the implication $\mathsf{OS}^* \to \mathsf{OS}$ only holds for types with infinitely many elements: by definition of standardness, a finite set has only standard elements.

Next, we deal with two characteristic principles of nonstandard Dialectica, whose validity in the filter topoi is independent of the metatheory.

\begin{prop}
The principle $\mathsf{NCR}$ holds in $\mathcal{N}$.
\end{prop}
\begin{proof}
Let $\mathcal{F}$ be any filter, $\Phi(z,x,y)$ an external formula, $x:S, y:T, z:U$, and $\alpha \in \llbracket U \rrbracket \mathcal{F}$. Assume $\mathcal{F} \Vdash \forall y:T \, \existst x:S \, \Phi(\alpha, x, y)\;$, or, equivalently,
\begin{equation*}
	\mathcal{F} \times T \Vdash \existst x:S \, \Phi(\alpha\pi_1, x, \pi_2)\;.
\end{equation*}
By the semantics of the $\existst$ quantifier in $\mathcal{N}$, this means that there exists $s \in S^*$ such that
\begin{equation*}
	\mathcal{F} \times T \Vdash \exists x\in s! \, \Phi(\alpha\pi_1, x, \pi_2)\;;
\end{equation*}
equivalently, since $s! = s!\pi_1$, $\mathcal{F} \Vdash \forall y:T \, \exists x \in s! \, \Phi(\alpha, x, y)$. Therefore,
\begin{equation*}
	\mathcal{F} \Vdash \existst s:S^* \, \forall y:T \, \exists x\in s! \, \Phi(\alpha, x, y)\;. \qedhere
\end{equation*}
\end{proof}

The next proof is a variant of one by Butz \cite{butz2004saturated}. It utilises the following, general result about Grothendieck topoi. Here, $\mathbf{a}$ is the sheafification functor.
\begin{lem} \label{lem:frommaclaneandmoerdijk}
Let $(\mathbf{C}, J)$ be a site. A set $\{f_i : C_i \to C\}_{i\in I}$ of morphisms of $\mathbf{C}$ is $J$-covering if and only if the set $\{\mathbf{ay}f_i: \mathbf{ay}C_i \to \mathbf{ay}C\}_{i \in I}$ is jointly epimorphic in $\mathrm{Sh}(\mathbf{C},J)$.
\end{lem}
\begin{proof}
See \cite[Corollary III.7.7]{mac1992sheaves}.
\end{proof}

\begin{prop}
The principle $\mathsf{HIP}_\forallst$ holds in $\mathcal{N}$.
\end{prop}
\begin{proof}
Let $\mathcal{F}$ be any filter, $\varphi(z,x)$ an internal formula, $\Psi(z,y)$ an external formula, with $x:S$, $y:T$, $z:U$, and $\alpha \in \llbracket U \rrbracket \mathcal{F}$. Suppose
\begin{equation*}
	\mathcal{F} \Vdash \forallst x:S \, \varphi(\alpha, x) \to \existst y:T \, \Psi(\alpha, y)\;.
\end{equation*}
By the semantics of first order logic in a Heyting category, this is equivalent to
\begin{equation*}
	\alpha^*\llbracket \forallst x \, \varphi(z,x) \rrbracket \leq \alpha^* \llbracket \existst y:T \, \Psi(z,y) \rrbracket
\end{equation*}
in $\mathrm{Sub}(\mathbf{y}\mathcal{F})$. By the semantics of the $\forallst$ predicate, we can write
\begin{equation*}
	\alpha^*\llbracket \forallst x \, \varphi(z,x) \rrbracket = \alpha^* \bigwedge_{x \in S} \llbracket \varphi(z,x!) \rrbracket\;;
\end{equation*}
and, by the suitable transfer theorem, for all $x \in S$,
\begin{equation*}
	\llbracket \varphi(z,x!) \rrbracket = \mathbf{y}\{z \in U \,|\, \varphi(z,x) \}\;.
\end{equation*}
Since the Yoneda embedding preserves and reflects all limits, we obtain
\begin{equation*}
	\alpha^*\llbracket \forallst x \, \varphi(z,x) \rrbracket = \mathbf{y}\Big( \alpha^* \bigwedge_{x \in S} \{z \in U \,|\, \varphi(z,x) \} \Big) =: \mathbf{y}\mathcal{H}\;.
\end{equation*}

For the consequence, we have, by the semantics of $\existst$ in $\mathcal{N}$, that
\begin{align*}
	\alpha^* \llbracket \existst y:T \, \Psi(z,y) \rrbracket & = \alpha^* \bigvee_{t\in T^*} \llbracket \exists y \in t! \, \Psi(z, y) \rrbracket = \\
	& = \bigvee_{t\in T^*} \alpha^* \llbracket \exists y \in t! \, \Psi(z, y) \rrbracket =: \bigvee_{t\in T^*} F_t\;,
\end{align*}
where we also used that unions are stable under pullback. Thus, there is a monomorphism $m: \mathbf{y}\mathcal{H} \rightarrowtail \bigvee_{t\in T^*} F_t$.

Let $\imath_t : F_t \rightarrowtail \bigvee_{t\in T^*} F_y$ be the inclusions of the $F_t$ in their union, for all $t \in T^*$, and consider the pullback diagrams
\begin{equation*}
\begin{tikzcd}
	m^* F_t \arrow[tail]{r} \arrow[tail]{d} & F_t \arrow[tail]{d}{\imath_t} \\
	\mathbf{y}\mathcal{H} \arrow[tail]{r}{m} & \bigvee_{t\in T^*} F_t \;.
\end{tikzcd}
\end{equation*}
Now, we use the fact that each $m^* F_t$ can be covered with a family of representable sheaves, to obtain a family $\{f_t : \mathbf{y}\mathcal{G}_t \to \mathbf{y}\mathcal{H} \}_{t \in T^*}$ of morphisms, such that each $m f_t$ factors through a \emph{single} $F_t$.

Moreover, since the $\{\imath_t : F_t \rightarrowtail \bigvee_{t\in T^*} F_t\}_{y \in T}$ jointly cover $\bigvee_{t\in T^*} F_t$, and in a Heyting pretopos all epimorphisms are stable under pullback \cite[Proposition IV.7.3]{mac1992sheaves}, the family $\{f_t : \mathbf{y}\mathcal{G}_t \to \mathbf{y}\mathcal{H} \}_{t \in T^*}$ is jointly epimorphic over $\mathbf{y}\mathcal{H}$.

By the previous lemma, we can extract from it a family of the form $\{\mathbf{y}\beta_k : \mathbf{y}\mathcal{G}_k \to \mathbf{y}\mathcal{H}\}_{k=1}^n$, where $\{\beta_k : \mathcal{G}_k \to \mathcal{H}\}_{k=1}^n$ is a $K$-cover in $\mathfrak{F}\mathbf{Set}$. Let $t := t_1 \cdot \ldots \cdot t_n$, such that $\mathbf{y}\beta_k$ factors through $F_{t_k}$, $k = 1,\ldots,n$. Then,
\begin{equation*}
	\mathbf{y}\mathcal{H} = \alpha^*\llbracket \forallst x \, \varphi(z,x) \rrbracket \leq \alpha^* \bigvee_{k=1}^n \llbracket \exists y \in t_k! \, \Psi(z, y) \rrbracket = \alpha^* \llbracket \exists y\in t! \, \Psi(z, y) \rrbracket\;.
\end{equation*}
Translating back to forcing semantics, this is precisely the statement that
\begin{equation*}
	\mathcal{F} \Vdash \forallst x:S \, \varphi(\alpha, x) \to \exists y\in t! \, \Psi(\alpha, y)\;,
\end{equation*}
from which it follows that
\begin{equation*}
	\mathcal{F} \Vdash \existst t:T^* \, \big(\forallst x:S \, \varphi(\alpha, x) \to \existst y\in t \, \Psi(\alpha, y)\big)\;. \qedhere
\end{equation*}
\end{proof}

So far, we used no principles whose constructive status is controversial, neither in the construction of the model, nor in our proofs. However, for our last pair of characteristic principles to hold, we must require that the \emph{axiom of choice} holds in the metatheory.

\begin{prop}
Suppose that the axiom of choice holds in the metatheory. Then the principle $\mathsf{HAC}^\mathrm{st}$ holds in $\mathcal{N}$.
\end{prop}
\begin{proof}
Let $\mathcal{F}$ be any filter, $\Phi(z,x,y)$ an external formula, $x:S, y:T, z:U$, and $\alpha \in \llbracket U \rrbracket \mathcal{F}$. Assume
\begin{equation*}
	\mathcal{F} \Vdash \forallst x:S \, \existst y:T \, \Phi(\alpha, x,y)\;;
\end{equation*}
by Lemma \ref{lem:quantifiers}, this means in $\mathcal{N}$ that, for all $x \in S$, there exists $t \in T^*$ such that
\begin{equation*}
	\mathcal{F} \Vdash \exists y \in t! \, \Phi(\alpha, x!, y)\;.
\end{equation*}
With the axiom of choice, we can find a function $f \in S \to T^*$ such that, for all $x \in S$,
\begin{equation*}
	\mathcal{F} \Vdash \exists y \in f(x)! \, \Phi(\alpha, x!, y)\;.
\end{equation*}
Since $\llbracket f(x)! \rrbracket = \llbracket f!(x!) \rrbracket$, it follows that $\mathcal{F} \Vdash \existst f:S \to T^* \, \forallst x:S \, \exists y \in f(x) \, \Phi(\alpha, x, y)$.
\end{proof}

In fact, a herbrandised version of the axiom of choice would suffice; but that would be a strange axiom to have in one's metatheory. The condition is necessary to a certain extent, for $\mathsf{HAC}^\mathrm{st}$ implies a herbrandised axiom of choice - call it $\mathsf{HAC}$ - in $\mathbf{Set}$: suppose
\begin{equation*}
	\forall x \in S \, \exists y \in T \, \varphi(x,y)\;.
\end{equation*}
By Corollary \ref{cor:trforallexists}, it follows that $\Vdash \forallst x:S \, \existst y:T \, \varphi(x,y)$. If $\mathsf{HAC}^\mathrm{st}$ holds in $\mathcal{N}$, we can deduce
\begin{equation*}
	\Vdash \existst f:S \to T^* \, \forallst x:S \, \exists y \in f(x) \, \varphi(x, y)\;;
\end{equation*}
and, applying the transfer theorem again, we obtain
\begin{equation*}
	\exists f \in S \to T^* \, \forall x \in S \, \exists y \in f(x) \, \varphi(x, y)
\end{equation*}
in $\mathbf{Set}$.

In the same way, the transfer rules can be used to rule out unconstrained validity of other principles in $\mathcal{N}$, as in the following example.

\begin{exm}
Let $T(s)$ be a binary tree, i.e.\ an internal formula on binary sequences such that
\begin{enumerate}
	\item $T(\langle \rangle)$ holds, and
	\item $\forall n, m \in \mathbb{N} \, \forall s \in 2^\mathbb{N} \, (T(\bar{s}m) \land n \leq m ) \to T(\bar{s}n)$, where $\bar{s}n := \langle s_0, \ldots, s_{n-1} \rangle$.
\end{enumerate}
The \emph{fan theorem} is the statement that for any such $T$, if, for all sequences $s \in 2^\mathbb{N}$, there exists $n \in \mathbb{N}$ such that $\neg \, T(\bar{s}n)$, then there exists some $n \in \mathbb{N}$ such that $\neg \, T(\bar{s}n)$ holds for all $s \in 2^\mathbb{N}$.

We consider the following, external version of the fan theorem:
\begin{equation*}
	\mathsf{FAN}^\mathrm{st}: \quad \forallst s:2^\mathbb{N} \, \existst n:\mathbb{N}\, \neg \, T(\bar{s}n) \to \existst n:\mathbb{N}\, \forallst s:2^\mathbb{N} \, \neg \, T(\bar{s}n)\;.
\end{equation*}
We claim that, if $\mathsf{FAN}^\mathrm{st}$ holds in $\mathcal{N}$, then the fan theorem holds in the metatheory. For suppose that, for all $s \in 2^\mathbb{N}$, there exists $n \in \mathbb{N}$ such that $\neg \, T(\bar{s}n)$. By transfer,
\begin{equation*}
	\Vdash \forallst s:2^\mathbb{N} \, \existst n:\mathbb{N}\, \neg \, T(\bar{s}n)\;;
\end{equation*}
and, if $\mathsf{FAN}^\mathrm{st}$ holds, we deduce
\begin{equation*}
	\Vdash \existst n:\mathbb{N}\,\forallst s:2^\mathbb{N} \, \neg \, T(\bar{s}n)\;.
\end{equation*}
This means that there exists a finite sequence $t$ of natural numbers, such that
\begin{equation*}
	\Vdash \exists n \in t! \, \forallst s:2^\mathbb{N} \, \neg \, T(\bar{s}n)\;.
\end{equation*}
By condition 2 on binary trees, we have that, if $\neg \, T(\bar{s}n)$ and $m \geq n$, then also $\neg \, T(\bar{s}m)$; therefore, picking $\tilde{n} := \max\{t_0, \ldots, t_{|t|-1}\}$, we are sure that
\begin{equation*}
	\Vdash \forallst s:2^\mathbb{N} \, \neg \, T(\bar{s}\tilde{n}!)\;.
\end{equation*}
By transfer, for all $s \in 2^\mathbb{N}$, $\neg \, T(\bar{s}\tilde{n})$, and we have proved the fan theorem.
\end{exm}

\section{The uniform Diller-Nahm interpretation} \label{sec:uniform}
In this section, we take a step back, forgetting about nonstandard arithmetic for a while; a reconsideration of ideas from Lifschitz, Berger, and Hernest leads us to a new functional interpretation - uniform Diller-Nahm - of which nonstandard Dialectica can be seen, \emph{a posteriori}, as a \emph{herbrandised} version.

\subsection{Calculability and herbrandisation}
In \cite{lifschitz1985calculable}, Lifschitz made the suggestion to see constructive mathematics as an \emph{extension} of classical mathematics.
Lifschitz's proposal is to enrich the language of Heyting arithmetic with a predicate $\mathrm{K}(n)$, ``$n$ is calculable''; and then extend Kleene's recursive realisability relation, write it $x \;\mathsf{r}_\mathrm{K}\; \varphi$, with the clause
\begin{itemize}
	\item[$\triangleright$] $x \;\mathsf{r}_\mathrm{K}\;\mathrm{K}(n)$ if and only if $x = n$,
\end{itemize}
all the while interpreting quantifiers \emph{uniformly}:
\begin{itemize}
	\item[$\triangleright$] $x \;\mathsf{r}_\mathrm{K}\; \forall n \, \varphi(n)$ if and only if $\forall n \, (x \;\mathsf{r}_\mathrm{K}\; \varphi(n))$,
	\item[$\triangleright$] $x \;\mathsf{r}_\mathrm{K}\; \exists n \, \varphi(n)$ if and only if $\exists n \, (x \;\mathsf{r}_\mathrm{K}\; \varphi(n))$.
\end{itemize}
In this definition, quantifiers are, by themselves, completely void of any computational meaning; it is by invoking quantifiers \emph{restricted to calculable numbers}, $\forall n\,(\mathrm{K}(n) \to \ldots)$, and $\exists n\, (\mathrm{K}(n) \land \ldots)$, that one restores it.

A couple of decades later, Lifschitz's demand was rediscovered, from a completely different perspective, in the area of proof mining. Rather than the foundational issue of injecting a certain ``modular constructiveness'' into classical reasoning, it was the practical problem of more efficient \emph{program extraction} from proofs that was addressed.

Even in fully intuitionistic proofs, a fine-grained analysis reveals instances of formulae with quantifiers that are \emph{computationally redundant}; i.e. the constructive content that is encoded in the quantifiers is never used in the program extracted with the aid of a functional interpretation. This always happens, in particular, when - in a natural deduction setting - an implication introduction discharges more then one instance of the same formula, so that the contraction rule needs to be used.

One would want a way to flag such quantifiers, telling the extraction program to just ``pass through'' them. This is the function performed by Berger's \emph{uniform} quantifiers \cite{berger2005uniform} and by Hernest's quantifiers \emph{without computational meaning} \cite{hernest2005light}. But, realisability being a rudimentary functional interpretation - this is clear, in particular, for Kreisel's \emph{modified} brand, see \cite{oliva2006unifying} - this is also what Lifschitz's calculability predicate achieved!

One possibly unexpected consequence of Lifschitz's ideas is that there will be two types of disjunction as well. One is a computationally empty disjunction $\lor$, with
\[ \Phi \lor \Psi \leftrightarrow \exists z:0 \, (z = 0 \to \Phi \land \neg \, z = 0 \to \Psi), \]
while there is also a computationally relevant disjunction $\lor_\mathrm{K}$, with
\[ \Phi \lor_\mathrm{K} \Psi \leftrightarrow \exists z:0 \, (\mathrm{K}(z) \land z = 0 \to \Phi \land \neg \, z = 0 \to \Psi). \]
These are not equivalent: in fact, only the second computationally relevant disjunction will act as a disjunction with respect to all the formulae in the language; the computationally empty disjunction only acts as a disjunction with respect to ``internal'' formulae (i.e.\ those not containing the $\mathrm{K}$-predicate).

\emph{Herbrandisation} can be seen as a way of repairing this schism. The idea is to weaken the computational meaning of the $\mathrm{K}$-predicate and define instead
\begin{itemize}
	\item[$\triangleright$] $x \;\mathsf{r}_\mathrm{K}\;\mathrm{K}(n)$ if and only if $x$ codes a sequence and $n$ is one of the components of the sequence coded by $x$.
\end{itemize}
This is reminiscent of \emph{Herbrand disjunctions} in classical logic - whence the name. There are some technical difficulties to overcome and this idea works especially well in the context of modified realisability, leading to \emph{Herbrand realisability} as introduced in \cite{van2012functional}. 

This process of herbrandisation is reflected in many of the characteristic principles of Herbrand realisability. While the axiom of choice for finite types
\begin{equation*}
	\mathsf{AC}: \quad \forall x:\sigma \, \exists y:\tau \, \Phi(x,y) \to \exists f:\sigma\to\tau \, \forall x:\sigma \, \Phi(x,fx)
\end{equation*}
is a characteristic principle of modified realisability, the herbrandised axiom of choice
\begin{equation*}
	\mathsf{HAC}^\mathrm{st}: \quad \forallst x:\sigma \, \existst y:\tau \, \Phi(x,y) \to \existst f:(\sigma\to\tau^*)^* \, \forallst x:\sigma \, \exists y \in f[x] \, \Phi(x,y)
\end{equation*}
(writing $\mathrm{st}$ again instead of $\mathrm{K}$) is a characteristic principle of Herbrand realisability. It seems natural to regard many of the characteristic principles of nonstandard Dialectica as herbrandisations of other, unherbrandised, principles, suggesting that also nonstandard Dialectica can be obtained by a process of herbrandisation from a functional interpretation which incorporates many of Lifschitz's ideas. The aim of this section is to show that this is indeed the case.

\begin{remark} Note that it is an immediate consequence of herbrandisation that disjunction loses any constructive meaning. In fact, a good way to think about herbrandisation is as a way of weakening the computational meaning of the $\mathrm{K}$-predicate in such a way that $\lor_\mathrm{K}$ collapses to the ordinary, computationally empty, disjunction.

Something like this is presumably unavoidable when one wants to interpret nonstandard systems: indeed, there seems to be a clash between the computational meaning of disjunction and nonstandard arithmetic. One way in which this manifests itself is that systems for nonstandard arithmetic often do not have the disjunction property: for example, E-HA$^{\omega*}_\mathrm{st} + \mathsf{OS}_0$ does not have the disjunction property, as proved in \cite{avigad2002transfer}. Another manifestation is the incompatibility of Church's Thesis for disjunctions
\[ \begin{array}{l} \forall x:0 \, \, \big( \, \varphi(x) \lor \psi(x) \,\big) \to \exists f:0 \to 0 \, \big( \, f \mbox{ is computable } \land \\
 \forall x:0 \, ( \, f(x) = 0 \to \varphi(x) \land f(x) \not= 0 \to \psi(x) \, ) \, \big) \end{array} \]
with the existence of nonstandard models for arithmetic (see \cite{mccarty1987variations}), showing, for instance, that there are no nonstandard models of arithmetic in the effective topos. This should be compared with the Herbrand topos from \cite{van2011herbrand}, where nonstandard models of arithmetic do exist.
\end{remark}

Before defining our new functional interpretation, we should first ``de-herbrandise'' our system. By the previous discussion, we can already guess that ``internal formula'' has to be replaced by the next best thing - ``internal \emph{and $\lor$-free} formula''.
\begin{itemize} \item[] \textbf{Notation.} If $\mathsf{P}$ is an axiom schema where certain schematic variables range over \emph{internal} formulae of E-HA$^{\omega*}_{\mathrm{st}}$, we write $\mathsf{P}_\lor$ for the same axiom schema, where ``internal'' is replaced by ``internal and $\lor$-free''.
\end{itemize}
Our tentative characteristic system is then E-HA$^{\omega*}_{\mathrm{st}\lor}$, that is the system E-HA$^{\omega*}_\mathrm{st}$ with $\mathsf{IA}_\lor$ in place of $\mathsf{IA}$, plus the characteristic principles $\mathsf{OS}^*_\lor$, $\mathsf{US}^*_\lor$, $\mathsf{AC}^\mathrm{st}$,
\begin{align*}
	& \mathsf{IP}_{\forall \, \lor}^\mathrm{st}: \quad (\forallst x:\sigma \, \varphi(x) \to \existst y:\tau \, \Psi(y)) \to \existst y:\tau \, (\forallst x:\sigma \, \varphi(x) \to \Psi(y))\;, \\
	& \mathsf{NU}: \quad \forall y: \tau\, \existst x: \sigma \, \Phi(x,y) \to \existst x: \sigma \, \forall y: \tau\, \Phi(x,y)\;.
\end{align*}
Since the restriction to $\lor$-free formulae also applies to the internal induction schema $\mathsf{IA}$, we do not get a proper system of arithmetic. This would actually be inconsistent with the nonstandard uniformity principle $\mathsf{NU}$.

\begin{prop}
In \emph{E-HA}$^{\omega*}_{\mathrm{st}\lor}$, the principle $\mathsf{NU}$ implies
\begin{equation*}
	\neg \, \forall n:0 \, (n = 0 \lor \neg \, n = 0)\;.
\end{equation*}
\end{prop}
\begin{proof}
Suppose $\forall n:0 \, (n = 0 \lor \neg \, n = 0)$. This is equivalent to
\begin{equation*}
	\forall n:0 \, \existst z:0 \, (z = 0 \to n = 0 \land \neg \, z = 0 \to \neg \, n = 0)\;,
\end{equation*}
which, by nonstandard uniformity, implies
\begin{equation*}
	\existst z:0 \, \forall n:0 \, (z = 0 \to n = 0 \land \neg \, z = 0 \to \neg \, n = 0)\;,
\end{equation*}
the statement that all natural numbers are zero, or all are non-zero; a contradiction.
\end{proof}
With the interpretation of $\mathrm{st}(n)$ as ``$n$ is calculable'', this is not unexpected; for how could we know whether a non-calculable number is zero or non zero? Notice that $\forallst n:0 \, (n = 0 \lor \neg \, n = 0)$ is still provable, thanks to the external induction schema.

The reason why we called $\mathsf{NU}$ a \emph{uniformity} principle is the similarity of
\begin{equation*}
	\forall s: 0^*\, \existst n: 0 \, \Phi(s,n) \to \existst n: 0 \, \forall s: 0^*\, \Phi(s,n)\;
\end{equation*}
to Troelstra's \emph{uniformity principle} \cite[Proposition 8.21]{troelstra1998chapter}
\begin{equation*}
	\mathsf{UP}: \quad \forall S \subseteq \mathbb{N} \, \exists n \in \mathbb{N} \, \Phi(S,n) \to \exists n \in \mathbb{N} \, \forall S \subseteq \mathbb{N} \, \Phi(S,n)\;,
\end{equation*}
a second-order principle that is validated by higher-order versions of recursive realisability, and which also has nonclassical consequences.

We can now define our de-herbrandised functional interpretation, prove that it is sound, and characterised by the desired proof system.

\subsection{The $U$ translation}
\begin{dfn} To every formula $\Phi(\tuple{a})$ of the language of E-HA$^{\omega*}_{\mathrm{st}\lor})$, with free variables $\tuple{a}$, we associate inductively its \emph{uniform Diller-Nahm} translation $\Phi(\tuple{a})^U = \existst \tuple{x}\,\forallst \tuple{y} \, \varphi_U(\tuple{x},\tuple{y},\tuple{a})$, where $\varphi_U$ is internal and $\lor$-free.
\begin{itemize}
	\item[$\triangleright$] $\varphi(\tuple{a})^U := \varphi_U(\tuple{a}) := \varphi(\tuple{a})$, for $\varphi$ internal atomic;
	\item[$\triangleright$] $\mathrm{st}_\sigma(x)^U := \existst y:\sigma \, ( y = x )\;$.
\end{itemize}
Let $\Phi(\tuple{a})^U = \existst \tuple{x}\,\forallst \tuple{y} \, \varphi_U(\tuple{x},\tuple{y},\tuple{a})$, $\Psi(\tuple{b})^U = \existst \tuple{u}\,\forallst \tuple{v} \, \psi_U(\tuple{u},\tuple{v},\tuple{b})$:
\begin{itemize}
	\item[$\triangleright$] $(\Phi(\tuple{a}) \land \Psi(\tuple{b}))^U := \existst \tuple{x},\tuple{u}\,\forallst \tuple{y},\tuple{v}\,\big(\varphi_U(\tuple{x},\tuple{y},\tuple{a}) \land \psi_U(\tuple{u},\tuple{v},\tuple{b})\big)\;$;
	\item[$\triangleright$] $(\Phi(\tuple{a}) \lor \Psi(\tuple{b}))^U := \existst z:0,\tuple{x},\tuple{u}\,\forallst \tuple{y},\tuple{v}\,\big(z = 0 \to \varphi_U(\tuple{x},\tuple{y},\tuple{a}) \land \neg\,z = 0 \to \psi_U(\tuple{u},\tuple{v},\tuple{b})\big)\;$;
	\item[$\triangleright$] $(\Phi(\tuple{a}) \to \Psi(\tuple{b}))^U := \existst \tuple{U},\tuple{Y}\,\forallst \tuple{x},\tuple{v}\,\big(\forall \tuple{y} \in \tuple{Y}\tuple{x}\tuple{v} \, \varphi_U(\tuple{x},\tuple{y},\tuple{a}) \to \psi_U(\tuple{U}\tuple{x},\tuple{v},\tuple{b})\big)\;$;
	\item[$\triangleright$] $(\exists z \,\Phi(z,\tuple{a}))^U := \existst \tuple{x} \, \forallst \tuple{y} \, \exists z \, \forall \tuple{y}' \in \tuple{y} \, \varphi_U(\tuple{x},\tuple{y}',z,\tuple{a})\;$;
	\item[$\triangleright$] $(\forall z \,\Phi(z,\tuple{a}))^U := \existst \tuple{x} \, \forallst \tuple{y} \, \forall z \, \varphi_U(\tuple{x},\tuple{y},z,\tuple{a})\;$;
	\item[$\triangleright$] $(\existst z \,\Phi(z,\tuple{a}))^U := \existst z,\tuple{x} \, \forallst \tuple{y} \, \varphi_U(\tuple{x},\tuple{y},z,\tuple{a})\;$;
	\item[$\triangleright$] $(\forallst z \,\Phi(z,\tuple{a}))^U := \existst \tuple{X} \, \forallst \tuple{y},z \, \varphi_U(\tuple{X}z,\tuple{y},z,\tuple{a})\;$.
\end{itemize}
\end{dfn}
The first thing to notice is that, if this interpretation is restricted to formulae that contain only external quantifiers - or, if you prefer, everything is declared standard - it is the same as the usual Diller-Nahm translation. In fact, except for a minor change in the interpretation of the uniform existential quantifier, it is to the Diller-Nahm variant precisely what Hernest's light Dialectica interpretation is to Dialectica.

Secondly, the interpretation is \emph{idempotent}: formulae of the form
\begin{equation*}
	\existst \tuple{x}\,\forallst \tuple{y} \, \varphi(\tuple{x},\tuple{y},\tuple{a})
\end{equation*}
with $\varphi$ internal and $\lor$-free are interpreted as themselves, as shown by an easy induction on their structure. This is a feature that the $D_\mathrm{st}$-translation lacked, due to the clause for the $\existst$ quantifier.

We will now prove soundness of the interpretation. We will not handle everything explicitly, though: except those concerning the quantifiers, all the logical axioms and rules admit the same realisers as those for the Diller-Nahm interpretation.

We write E-HA$^{\omega*}_\lor$ for the system E-HA$^{\omega*}$ with $\mathsf{IA}_\lor$ in place of $\mathsf{IA}$.

\begin{thm}[Soundness of uniform Diller-Nahm]
Suppose
\begin{equation*}
	\text{\emph{E-HA}}^{\omega*}_{\mathrm{st}\lor} + \mathsf{OS}^*_\lor + \mathsf{US}^*_\lor + \mathsf{NU} + \mathsf{AC}^\mathrm{st} + \mathsf{IP}_{\forall \, \lor}^\mathrm{st} + \Delta_\lor \vdash \Phi(\tuple{a})\;,
\end{equation*}
where $\Delta_\lor$ is a set of internal, $\lor$-free sentences. Let $\Phi(\tuple{a})^U = \existst \tuple{x}\,\forallst \tuple{y} \, \varphi_U(\tuple{x},\tuple{y},\tuple{a})$. Then from the proof we can extract a tuple of closed terms $\tuple{t}$ such that
\begin{equation*}
	\text{\emph{E-HA}}^{\omega*}_\lor + \Delta_\lor \vdash \forall \tuple{y} \, \varphi_U(\tuple{t}, \tuple{y}, \tuple{a})\;.
\end{equation*}
\end{thm}
\begin{proof}
We proceed by induction on the length of the derivation.
\begin{enumerate}
\item \emph{The logical axioms and rules of intuitionistic first order predicate logic}. We consider the quantifier axioms and rules, and give another couple of examples, referring again to \cite[3.5.4]{troelstra1973metamathematical} for the rest.
	\begin{enumerate}[label=(\roman*)]
	\item Example - $\mathsf{weakening}: \quad A \to A \lor B$.
	
	Suppose $A^U = \existst \tuple{x}\,\forallst \tuple{y} \, \varphi(\tuple{x},\tuple{y},\tuple{a})$, $B^U = \existst \tuple{u}\,\forallst \tuple{v} \, \psi(\tuple{u},\tuple{v},\tuple{b})$. Then
	\begin{align*}
		(A \to A \lor B)^U = \; & \existst \tuple{Z}, \tuple{X}', \tuple{U}, \tuple{S} \, \forallst \tuple{x}, \tuple{y'}, \tuple{v} \, \big(\forall \tuple{y} \in \tuple{S}\tuple{x}\tuple{y}'\tuple{v} \, \varphi(\tuple{x},\tuple{y},\tuple{a}) \to \\
		& (\tuple{Z}\tuple{x} = 0 \to \varphi(\tuple{X}'\tuple{x}, \tuple{y}', \tuple{a}) \land \neg \, \tuple{Z}\tuple{x} = 0 \to \psi(\tuple{U}\tuple{x}, \tuple{v}, \tuple{b}))\big)\;,
	\end{align*}
	and we can take
	\begin{equation*}\begin{matrix*}[l]
		& \tuple{Z} := \lambda \tuple{x}.0\;, & \tuple{X}' := \lambda \tuple{x}.\tuple{x}\;, \\
		& \tuple{U} \text{ arbitrary,} & \tuple{S} := \lambda \tuple{x}, \tuple{y}', \tuple{v}.\langle \tuple{y}' \rangle\;.
	\end{matrix*}\end{equation*}
	
	\item $\forall z \, A \to A[b/z]$.
	
	Suppose $A^U = \existst \tuple{x}\,\forallst \tuple{y} \, \varphi(\tuple{x},\tuple{y},z, \tuple{a})$. Then
	\begin{align*}
		(\forall z \, A \to A[b/z])^U = \; & \existst \tuple{X}', \tuple{S} \, \forallst \tuple{x}, \tuple{y}' \, \\
		&\big(\forall \tuple{y} \in \tuple{S}\tuple{x}\tuple{y}' \, \forall z\, \varphi(\tuple{x},\tuple{y},z,\tuple{a}) \to \varphi(\tuple{X}'\tuple{x},\tuple{y}',b,\tuple{a}) \big)\;,
	\end{align*}
	so we can take
	\begin{equation*}\begin{matrix*}[l]
		& \tuple{X}' := \lambda \tuple{x}.\tuple{x}\;, & \tuple{S} := \lambda \tuple{x}, \tuple{y}'.\langle \tuple{y}' \rangle\;.
	\end{matrix*} \end{equation*}
	
	\item $A[b/z] \to \exists z \, A$.
	
	Suppose $A^U = \existst \tuple{x}\,\forallst \tuple{y} \, \varphi(\tuple{x},\tuple{y},z, \tuple{a})$. Then
	
	\begin{align*}
		(A[b/z] \to \exists z \, A)^U = \; & \existst \tuple{X}', \tuple{S} \, \forallst \tuple{x}, \tuple{t} \, \\
		& \big(\forall \tuple{y} \in \tuple{S}\tuple{x}\tuple{t} \, \varphi(\tuple{x}, \tuple{y}, b, \tuple{a}) \to \exists z \, \forall \tuple{y}' \in \tuple{t} \, \varphi(\tuple{X}'\tuple{x}, \tuple{y}', z, \tuple{a}) \big)\;,
	\end{align*}
	and we can take
	\begin{equation*}\begin{matrix*}[l]
		& \tuple{X}' := \lambda \tuple{x}.\tuple{x}\;, & \tuple{S} := \lambda \tuple{x}, \tuple{t}.\tuple{t}\;.
	\end{matrix*} \end{equation*}
	
	\item Example - $\mathsf{modus\;ponens}$.
	
	Suppose that $A^U = \existst \tuple{x}\,\forallst \tuple{y} \, \varphi(\tuple{x},\tuple{y},\tuple{a})$, $B^U = \existst \tuple{u}\,\forallst \tuple{v} \, \psi(\tuple{u},\tuple{v},\tuple{b})$, and that we have terms $\tuple{t}_1$ realising the interpretation of $A^U$ and $\tuple{T}_2, \tuple{T}_3$ realising the interpretation of $(A \to B)^U$.
	
	This means we have
	\begin{equation*}
		\text{E-HA}^{\omega*}_\lor + \Delta_\lor \vdash \forall \tuple{y} \, \varphi(\tuple{t}_1, \tuple{y}, \tuple{a})\;,
	\end{equation*}
	and
	\begin{equation*}
		\text{E-HA}^{\omega*}_\lor + \Delta_\lor \vdash \forall \tuple{x},\tuple{v} \, \big( \forall \tuple{y} \in \tuple{T}_3 \tuple{x}\tuple{v} \, \varphi(\tuple{x}, \tuple{y}, \tuple{a}) \to \psi(\tuple{T}_2\tuple{x}, \tuple{v}, \tuple{b})\big)\;.
	\end{equation*}
	Taking $\tuple{t}_4 := \tuple{T}_2\tuple{t}_1$, we obtain
	\begin{equation*}
		\text{E-HA}^{\omega*}_\lor + \Delta_\lor \vdash \forall \tuple{v} \, \psi(\tuple{t}_4, \tuple{v}, \tuple{b})\;,
	\end{equation*}
	as desired.
	
	\item $\begin{array}{c}B \to A \\ \hline B \to \forall z \, A \end{array}\;$.
	
	Suppose that $A^U = \existst \tuple{x}\,\forallst \tuple{y} \, \varphi(\tuple{x},\tuple{y},z,\tuple{a})$, $B^U = \existst \tuple{u}\,\forallst \tuple{v} \, \psi(\tuple{u},\tuple{v},\tuple{b})$, where $z$ is not free in $\psi$, and that we have terms $\tuple{T}_1, \tuple{T}_2$ realising $(B \to A)^U$. Then,
	
	 \begin{equation*}
		\text{E-HA}^{\omega*}_\lor + \Delta_\lor \vdash \forall \tuple{u},\tuple{y} \, \big( \forall \tuple{v} \in \tuple{T}_2 \tuple{u}\tuple{y} \, \psi(\tuple{u}, \tuple{v}, \tuple{b}) \to \varphi(\tuple{T}_1\tuple{u}, \tuple{y}, z, \tuple{a})\big)\;.
	\end{equation*}
	Then $\tuple{T}_3 := \tuple{T}_1$ and $\tuple{T}_4 := \tuple{T}_2$ realise the interpretation of $B \to \forall z \, A$.
	
	\item $\begin{array}{c}A \to B \\ \hline \exists z \, A \to B  \end{array}\;$.
	
	Suppose that $A^U = \existst \tuple{x}\,\forallst \tuple{y} \, \varphi(\tuple{x},\tuple{y},z,\tuple{a})$, $B^U = \existst \tuple{u}\,\forallst \tuple{v} \, \psi(\tuple{u},\tuple{v},\tuple{b})$, where $z$ is not free in $\psi$, and that we have terms $\tuple{T}_1, \tuple{T}_2$ realising $(A \to B)^U$. Then,
	\begin{equation*}
		\text{E-HA}^{\omega*}_\lor + \Delta_\lor \vdash \forall \tuple{x},\tuple{v} \, \big( \forall \tuple{y} \in \tuple{T}_2 \tuple{x}\tuple{v} \, \varphi(\tuple{x}, \tuple{y}, z, \tuple{a}) \to \psi(\tuple{T}_1\tuple{x}, \tuple{v}, \tuple{b})\big)\;.
	\end{equation*}
	
	We have
	\begin{align*}
		(\exists z\, A \to B)^U = \; & \existst \tuple{U}, \tuple{S} \, \forallst \tuple{x},\tuple{v} \, \\
		& \big(\forall s \in \tuple{S}\tuple{x}\tuple{v} \, \exists z \, \forall y \in s \, \varphi(\tuple{x},\tuple{y}, z, \tuple{a}) \to \psi(\tuple{U}\tuple{x}, \tuple{v}, \tuple{b})\big)\;;
	\end{align*}
	so we can take $\tuple{T}_3 := \tuple{T}_1$, and $\tuple{T}_4 := \lambda \tuple{x},\tuple{v}.\langle \tuple{T}_2\tuple{x}\tuple{v} \rangle\;$, to obtain
	\begin{equation*}
		\text{E-HA}^{\omega*}_\lor + \Delta_\lor \vdash \forall \tuple{x},\tuple{v} \, \big(\forall s \in \tuple{T}_4\tuple{x}\tuple{v} \, \exists z \, \forall y \in s \, \varphi(\tuple{x},\tuple{y}, z, \tuple{a}) \to \psi(\tuple{T}_3\tuple{x}, \tuple{v}, \tuple{b})\big)\;.
	\end{equation*}
	\end{enumerate}
	
\item \emph{The nonlogical axioms of extensional Heyting arithmetic in all finite types} (with the restricted induction schema $\mathsf{IA}_\lor$). These are all internal and $\lor$-free, hence are realised by the empty tuple.
	
\item \emph{The defining axioms of the external quantifiers}. Let $\Phi(x)^U := \existst \tuple{u} \, \forallst \tuple{v} \, \varphi(\tuple{u}, \tuple{v}, x)$ here.
	\begin{enumerate}[label=(\roman*)]
	\item $\forallst x \,\Phi(x) \leftrightarrow \forall x \,(\mathrm{st}(x) \to \Phi(x))\;$.
	
	The interpretation of $\forallst x \,\Phi(x) \to \forall x \,(\mathrm{st}(x) \to \Phi(x))$ is
	\begin{align*}
	& \existst \tuple{U'}, S, \tuple{T} \, \forallst \tuple{U}, y, \tuple{v}' \,\big(\forall x \in S\tuple{U}y\tuple{v}' \, \forall \tuple{v} \in \tuple{T}\tuple{U}y\tuple{v}' \, \varphi(\tuple{U}x, \tuple{v}, x) \to \\
	& \forall x \,(x = y \to \varphi(\tuple{U}'\tuple{U}y, \tuple{v}', x) )\big)\;;
	\end{align*}
	so we can take
	\begin{equation*}\begin{matrix*}[l]
		& \tuple{U}' := \lambda \tuple{U}, y.\tuple{U}y\;, & S := \lambda \tuple{U},y,\tuple{v}'.\langle y \rangle\;, \\
		& \tuple{T} := \lambda \tuple{U}, y, \tuple{v}'.\langle \tuple{v}' \rangle \;.
	\end{matrix*}\end{equation*}
	
	On the other hand, the interpretation of $\forall x \,(\mathrm{st}(x) \to \Phi(x)) \to \forallst x \,\Phi(x)$ is
	\begin{align*}
	& \existst \tuple{U}', S, \tuple{T} \, \forallst x', \tuple{U}, \tuple{v}' \, \big(\forall y \in Sx'\tuple{U}\tuple{v}' \; \forall \tuple{v} \in \tuple{T} x' \tuple{U} \tuple{v}' \; \forall x \, \\
	& (x = y \to \varphi(\tuple{U}y, \tuple{v}, x)) \to \varphi(\tuple{U}'\tuple{U}x', \tuple{v}', x')\big)\;,
	\end{align*}
	and we can take
	\begin{equation*}\begin{matrix*}[l]
		& \tuple{U}' := \lambda \tuple{U}, x'.\tuple{U}x'\;, & S := \lambda x',\tuple{U},\tuple{v}'.\langle x' \rangle\;, \\
		& \tuple{T} := \lambda x', \tuple{U}, \tuple{v}'.\langle \tuple{v}' \rangle \;.
	\end{matrix*}\end{equation*}
	
	\item $\existst x \,\Phi(x) \leftrightarrow \exists x \,(\mathrm{st}(x) \land \Phi(x))\;$.
	
	The interpretation of $\existst x \,\Phi(x) \to \exists x \,(\mathrm{st}(x) \land \Phi(x))$ is
	\begin{align*}
	& \existst Y, \tuple{U}', \tuple{T} \, \forallst x, \tuple{u}, \tuple{s} \, \big(\forall \tuple{v} \in \tuple{T}x\tuple{u}\tuple{s} \; \varphi(\tuple{u},\tuple{v},x) \to \\
	& \exists x' \, \forall \tuple{v}' \in \tuple{s} \, (Yx\tuple{u} = x' \land \varphi(\tuple{U}'x\tuple{u}, \tuple{v}', x'))\big)\;;
	\end{align*}
	so we can take
	\begin{equation*}\begin{matrix*}[l]
		& Y := \lambda x,\tuple{u}.x\;, & \tuple{U}' := \lambda x,\tuple{u}.\tuple{u}\;, \\
		& \tuple{T} := \lambda x,\tuple{u},\tuple{s}.\tuple{s} \;.
	\end{matrix*}\end{equation*}
	
	The interpretation of its converse $\exists x \,(\mathrm{st}(x) \land \Phi(x)) \to \existst x \,\Phi(x)$ is
	\begin{align*}
	& \existst X, \tuple{U}, \tuple{S} \, \forallst y, \tuple{u}', \tuple{v} \,\big(\forall \tuple{s} \in \tuple{S}y\tuple{u}'\tuple{v} \; \exists x' \; \forall \tuple{v}' \in \tuple{s} \, \\
	& (y = x' \land \varphi(\tuple{u}', \tuple{v}', x')) \to \varphi(\tuple{U}y\tuple{u}', \tuple{v}, Xy\tuple{u}')\big)\;,
	\end{align*}
	and we can take
	\begin{equation*}\begin{matrix*}[l]
		& X := \lambda y,\tuple{u}'.y\;, & \tuple{U} := \lambda y,\tuple{u}'.\tuple{u}'\;, \\
		& \tuple{S} := \lambda y,\tuple{u}',\tuple{v}.\langle \langle \tuple{v} \rangle \rangle \;.
	\end{matrix*}\end{equation*}
	\end{enumerate}
	
\item \emph{The axioms for the standardness predicate}.
	\begin{enumerate}[label=(\roman*)]
	\item $\mathrm{st}(x) \land x = y \to \mathrm{st}(y)\;$.
	
	The interpretation of this axiom is
	\begin{equation*}
		\existst Y' \, \forallst x' \,(x = x' \land x = y \to y = Y'x')\;,
	\end{equation*}
	so we can take $Y' := \lambda x'. x'\;$.
	
	\item $\mathrm{st}(a)$ for all closed terms $a$.
	
	We have $(\mathrm{st}(a))^U = \existst x \, (a = x)$, so we can take $x := a$.

	\item $\mathrm{st}(f) \land \mathrm{st}(x) \to \mathrm{st}(fx)$.
	
	The interpretation of this axiom is
	\begin{equation*}
		\existst Y \, \forallst f', x' \,(f = f' \land x = x' \to fx = Yf'x')\;,
	\end{equation*}
	so we can take $Y := \lambda f',x'.f'x'\;$.
	\end{enumerate}
	
\item \emph{The external induction schema}.

As in \cite{van2012functional}, we consider the equivalent external induction \emph{rule}
\begin{equation*}
	\mathsf{IR}^\mathrm{st}: \quad \begin{array}{c}\Phi(0) \quad \forallst n:0 \,(\Phi(n) \to \Phi(n+1)) \\ \hline \forallst n:0 \, \Phi(n) \end{array}\;,
\end{equation*}
from which the external induction schema is obtained by taking $\Phi(m) := \Psi(0) \land \forallst n:0 \, (\Psi(n) \to \Psi(n+1)) \to \Psi(m)\;$.

So, suppose that $(\Phi(n))^U = \existst \tuple{x}\,\forallst \tuple{y} \, \varphi(\tuple{x},\tuple{y},n,\tuple{a})$, and that we have realisers $\tuple{t}_1$, and $\tuple{T}_2, \tuple{T}_3$ for the premises; i.e.
\begin{equation*}
	\text{E-HA}^{\omega*}_\lor + \Delta_\lor \vdash \forall \tuple{y} \, \varphi(\tuple{t}_1, \tuple{y}, 0, \tuple{a})\;,
\end{equation*}
and
\begin{equation*}
	\text{E-HA}^{\omega*}_\lor + \Delta_\lor \vdash \forall n, \tuple{x}, \tuple{y}' \,(\forall \tuple{y} \in \tuple{T}_3 n\tuple{x}\tuple{y}' \, \varphi(\tuple{x},\tuple{y},n, \tuple{a}) \to \varphi(\tuple{T}_2 n\tuple{x}, \tuple{y}', n+1, \tuple{a}))\;.
\end{equation*}

By taking $\tuple{T}_4 := \lambda n.\mathrm{R}\tuple{t}_1 \tuple{T}_2 n$, we obtain, by induction for $\lor$-free formulae in E-HA$^{\omega*}_\lor$, that
\begin{equation*}
	\text{E-HA}^{\omega*}_\lor + \Delta_\lor \vdash \forall n, \tuple{y} \, \varphi(\tuple{T}_4 n, \tuple{y}, n, \tuple{a})\;,
\end{equation*}
which was to be proved.

\item \emph{The principles} $\mathsf{OS}^*_\lor, \mathsf{US}^*_\lor, \mathsf{NU}, \mathsf{AC}^\mathrm{st}, \mathsf{IP}_{\forall \, \lor}^\mathrm{st}$.
	\begin{enumerate}[label=(\roman*)]
	\item $\mathsf{OS}^*_\lor: \quad \forallst s \, \varphi(s) \to \exists s \, (\forallst x \, (x \in s) \land \varphi(s))\;$, with $\varphi$ internal and $\lor$-free.
	
	This is interpreted as
	\begin{equation*}
		\existst S \, \forallst s' \,\big(\forall s \in Ss' \, \varphi(s) \to \exists s\,(s' \subseteq s \land \varphi(s))\big)\;,
	\end{equation*}
	and we can take $S := \lambda s'.\langle s' \rangle\;$.
	
	\item $\mathsf{US}^*_\lor: \quad \forall s \, (\forallst x \, (x \in s) \to \varphi(s)) \to \existst s \, \varphi(s)\;$, with $\varphi$ internal and $\lor$-free.
	
	The interpretation of this axiom schema is
	\begin{equation*}
		\existst S \, \forallst s' \,\big(\forall s \,(s' \subseteq s \to \varphi(s)) \to \varphi(Ss')\big)\;;
	\end{equation*}
	so we can take $S := \lambda s'. s'\;$.
	\end{enumerate}
For the principles $\mathsf{NU}, \mathsf{AC}^\mathrm{st}, \mathsf{IP}_{\forall \, \lor}^\mathrm{st}$, we can just observe that the premise and the conclusion have identical interpretations, so it is trivial to find realisers for the implication. We do the first as an example.
	\begin{enumerate}[label=(\roman*)]
	\item[(iii)] $\mathsf{NU}: \quad \forall y \, \existst x \, \Phi(x,y) \to \existst x \, \forall y \, \Phi(x,y)\;$.
	
	Let $\Phi(x,y)^U := \existst \tuple{u} \, \forallst \tuple{v} \, \varphi(\tuple{u}, \tuple{v}, x, y)$. Both the premise and the conclusion are interpreted as
	\begin{equation*}
		\existst x, \tuple{u} \, \forallst \tuple{v} \, \forall y \, \varphi(\tuple{u}, \tuple{v}, x, y)\;;
	\end{equation*}
	so the implication is interpreted as
	\begin{equation*}
		\existst X', \tuple{U}', \tuple{S} \, \forallst x, \tuple{u}, \tuple{v}' \, \big(\forall \tuple{v} \in \tuple{S}x\tuple{u}\tuple{v}' \, \forall y \, \varphi(\tuple{u},\tuple{v}, x, y) \to \forall y \, \varphi(\tuple{U}'x\tuple{u}, \tuple{v}', X'x\tuple{u}, y)\big)\;,
	\end{equation*}
	and we can take
	\begin{equation*}\begin{matrix*}[l]
		& X' := \lambda x,\tuple{u}.x\;, & \tuple{U}' := \lambda x,\tuple{u}.\tuple{u}\;, \\
		& \tuple{S} := \lambda x,\tuple{u},\tuple{v}'.\langle \tuple{v}' \rangle \;.
	\end{matrix*}\end{equation*}
	\end{enumerate}
\end{enumerate}
This concludes the proof. \end{proof}

As usual, soundness of an interpretation leads to a conservation result.
\begin{cor} \label{cor:conserv_u}
The system
\begin{equation*}
	\text{\emph{E-HA}}^{\omega*}_{\mathrm{st}\lor} + \mathsf{OS}^*_\lor + \mathsf{US}^*_\lor + \mathsf{NU} + \mathsf{AC}^\mathrm{st} + \mathsf{IP}_{\forall \, \lor}^\mathrm{st}
\end{equation*}
is conservative with respect to $\lor$-free formulae of \emph{E-HA}$^{\omega*}_\lor$.
\end{cor}
\begin{proof}
Follows immediately from the previous theorem.
\end{proof}

Now, let
\begin{equation*}
	\text{H} := \text{E-HA}^{\omega*}_{\mathrm{st}\lor} + \mathsf{OS}^*_\lor + \mathsf{US}^*_\lor + \mathsf{NU} + \mathsf{AC}^\mathrm{st} + \mathsf{IP}_{\forall \, \lor}^\mathrm{st}\;.
\end{equation*}

\begin{thm}[Characterisation of uniform Diller-Nahm] \label{thm:char_udn}

Let $\Phi$ be a formula in the language of $\text{\emph{E-HA}}^{\omega*}_{\mathrm{st}\lor}$.
\begin{enumerate}[label=(\alph*)]
	\item $\text{\emph{H}} \vdash \Phi \leftrightarrow \Phi^U\;$.
	\item If for all formulae $\Psi$ of the language of \emph{E-HA}$^{\omega*}_{\mathrm{st}\lor})$, with $\Psi^U = \existst \tuple{x} \, \forallst \tuple{y} \, \psi(\tuple{x}, \tuple{y})$,
	\begin{equation*}
		\text{\emph{H}} + \Phi \vdash \Psi
	\end{equation*}
	implies that there exist closed terms $\tuple{t}$ such that
	\begin{equation*}
		\text{\emph{E-HA}}^{\omega*}_\lor \vdash \forall \tuple{y} \, \psi(\tuple{t},\tuple{y})\;
	\end{equation*}
	holds, then $\text{\emph{H}} \vdash \Phi\,$.
\end{enumerate}
\end{thm}
\begin{proof}
We prove (a) by induction on the logical structure of $\Phi$. For $\Phi \equiv \varphi$ internal atomic, obviously $\text{H} \vdash \varphi \leftrightarrow \varphi^U$.
	
Let $\Phi \equiv \mathrm{st}(x)$. If $x$ is standard, it follows that $\existst y \, (x = y)$, by taking $y := x$. Conversely, if $\existst y \, (x = y)$, by the first axiom for the standardness predicate it follows that $x$ is standard. Hence,
\begin{equation*}
	\text{H} \vdash \mathrm{st}(x) \leftrightarrow \existst y \, (x = y)\;.
\end{equation*}

For the induction hypothesis, using an appropriate embedding of tuples of types into higher types, and a compatible coding of tuples of terms \cite[1.6.17]{troelstra1973metamathematical}, we can assume, given formulae $\Phi$ and $\Psi$, that there exist internal, $\lor$-free formulae $\varphi$, $\psi$ such that
\begin{align*}
	\text{H} & \vdash \Phi(x) \leftrightarrow \existst x \, \forallst y \, \varphi(x,y)\;, \\
	\text{H} & \vdash \Psi(x) \leftrightarrow \existst u \, \forallst v \, \psi(u,v)\;.
\end{align*}

\begin{enumerate}[label=(\roman*)]
	\item For $\land$, by intuitionistic logic,
	\begin{equation*}
		\existst x \, \forallst y \, \varphi(x,y) \land \existst u \, \forallst v \, \psi(u,v)
	\end{equation*}
	is equivalent to
	\begin{equation*}
		\existst x, u \, \forallst y,v \, (\varphi(x,y) \land \psi(u,v))\;.
	\end{equation*}
	
	\item For $\lor$,
	\begin{equation*}
		\existst x \, \forallst y \, \varphi(x,y) \lor \existst u \, \forallst v \, \psi(u,v)
	\end{equation*}
	is equivalent in H to
	\begin{equation*}
		\existst z:0 \, (z = 0 \to \existst x \, \forallst y \, \varphi(x,y) \land \neg\, z= 0 \to \existst u \, \forallst v \, \psi(u,v))\;.
	\end{equation*}
	By $\mathsf{IP}_{\forall \, \lor}^\mathrm{st}$, this is equivalent to
	\begin{equation*}
		\existst z:0 \, \big(\existst x \, \forallst y \,(z = 0 \to \varphi(x,y)) \land \existst u \, \forallst v \, (\neg\, z = 0 \to \psi(u,v))\big)\;,
	\end{equation*}
	and we are back to the case of conjunction.
	
	\item For $\to$, we proceed as with Diller-Nahm implication. By intuitionistic logic and the principle $\mathsf{IP}_{\forall \, \lor}^\mathrm{st}$,
	\begin{equation*}
		\existst x \, \forallst y \, \varphi(x,y) \to \existst u \, \forallst v \, \psi(u,v)
	\end{equation*}
	is equivalent to
	\begin{equation*}
		\forallst x \, \existst u \, \forallst v \, (\forallst y \, \varphi(x,y) \to \psi(u,v))\;.
	\end{equation*}
	Now, adapting Proposition \ref{prop:hgmpst}, we see that $\text{E-HA}^{\omega*}_{\mathrm{st}\lor} + \mathsf{US}^*_\lor \vdash \mathsf{HGMP}^\mathrm{st}_\lor$, so this is equivalent to
	\begin{equation*}
		\forallst x \, \existst u \, \forallst v \, \existst s \, (\forall y \in s \, \varphi(x,y) \to \psi(u,v))\;.
	\end{equation*}
	Two applications of $\mathsf{AC}^\mathrm{st}$ then lead to
	\begin{equation*}
		\existst U, S \, \forallst x, v \, (\forall y \in Sxv \, \varphi(x,y) \to \psi(Ux,v))\;.
	\end{equation*}
	
	\item For $\exists z$, adapting Proposition \ref{prop:ios}, we see that $\text{E-HA}^{\omega*}_{\mathrm{st}\lor} + \mathsf{OS}^*_\lor \vdash \mathsf{I}_\lor$; therefore
	\begin{equation*}
		\exists z \, \existst x \, \forallst y \, \varphi(x,y,z)
	\end{equation*}
	is equivalent to
	\begin{equation*}
		\existst x \, \forallst s \, \exists z \, \forall y \in s \, \varphi(x,y,z)\;.
	\end{equation*}
	
	\item For $\forall z$, we use that by $\mathsf{NU}$
	\begin{equation*}
		\forall z \, \existst x \, \forallst y \, \varphi(x,y,z)
	\end{equation*}
	is equivalent to
	\begin{equation*}
		\existst x \, \forallst y \, \forall z \, \varphi(x,y,z)\;.
	\end{equation*}
	
	\item For $\existst z$, nothing really needs to be done.
	
	\item For $\forallst z$, we just use $\mathsf{AC}^\mathrm{st}$ once in order to obtain that
	\begin{equation*}
		\forallst z \, \existst x \, \forallst y \, \varphi(x,y,z)
	\end{equation*}
	is equivalent to
	\begin{equation*}
		\existst X \, \forallst y,z \, \varphi(Xz,y,z)\;.
	\end{equation*}	
\end{enumerate}
This proves item (a). For (b), suppose $\Phi$ satisfies the condition, and that $\Phi^U = \existst \tuple{x} \, \forallst \tuple{y} \, \varphi(\tuple{x}, \tuple{y})$. Then, from
\begin{equation*}
	\text{H} + \Phi \vdash \Phi
\end{equation*}
it follows that there exist closed terms $\tuple{t}$ such that
\begin{equation*}
	\text{E-HA}^{\omega*}_\lor \vdash \forall \tuple{y} \, \varphi(\tuple{t},\tuple{y})\;.
\end{equation*}
From this, we obtain $\text{E-HA}^{\omega*}_{\mathrm{st}\lor} \vdash \forallst \tuple{y} \, \varphi(\tuple{t},\tuple{y})\;$, whence
\begin{equation*}
	\text{E-HA}^{\omega*}_{\mathrm{st}\lor} \vdash \existst \tuple{x} \, \forallst \tuple{y} \, \varphi(\tuple{x},\tuple{y})\;;
\end{equation*}
so $\text{H} \vdash \existst \tuple{x} \, \forallst \tuple{y} \, \varphi(\tuple{x},\tuple{y})\;$ as well. But then, by the equivalence of (a),
\begin{equation*}
	\text{H} \vdash \Phi
\end{equation*}
which was to be proved.
\end{proof}

We now show how the uniform Diller-Nahm interpretation may be used to extract programs from proofs, and eliminate instances of its characteristic principles.
\begin{thm}[Program extraction by the $U$ interpretation]
Let $\forallst x \, \existst y\, \varphi(x,y)$ be a sentence of $\text{\emph{E-HA}}^{\omega*}_{\mathrm{st}\lor}$, with $\varphi$ internal and $\lor$-free, and let $\Delta_\lor$ be a set of internal, $\lor$-free sentences. If
\begin{equation*}
	\text{\emph{E-HA}}^{\omega*}_{\mathrm{st}\lor} + \mathsf{OS}^*_\lor + \mathsf{US}^*_\lor + \mathsf{NU} + \mathsf{AC}^\mathrm{st} + \mathsf{IP}_{\forall \, \lor}^\mathrm{st} + \Delta_\lor \vdash \forallst x \, \existst y\, \varphi(x,y)\;,
\end{equation*}
then from the proof we can extract a closed term $T$ such that
\begin{equation*}
	\text{\emph{E-HA}}^{\omega*}_\lor + \Delta_\lor \vdash \forall x \, \varphi(x, Tx)\;.
\end{equation*}
\end{thm}
\begin{proof}
The $U$ translation of $\forallst x \, \existst y\, \varphi(x,y)$ is
\begin{equation*}
	\existst f \, \forallst x \, \varphi(x,fx)\;,
\end{equation*}
so the thesis immediately follows from the soundness theorem.
\end{proof}

Finally, we derive a few properties of the system H, which follow from the properties of the uniform Diller-Nahm interpretation.

\begin{prop}
The system $\text{\emph{E-HA}}^{\omega*}_{\mathrm{st}\lor} + \mathsf{OS}^*_\lor + \mathsf{US}^*_\lor + \mathsf{NU} + \mathsf{AC}^\mathrm{st} + \mathsf{IP}_{\forall \, \lor}^\mathrm{st}$ is closed under the restricted transfer rules
\begin{align*}
 	& \mathsf{TR}_{\forall\lor}: \quad \begin{array}{c} \forallst x:\sigma \, \varphi(x) \\ \hline \forall x:\sigma \, \varphi(x)  \end{array}\;, \\
 	& \mathsf{TR}_{\exists\lor}: \quad \begin{array}{c} \exists x:\sigma \, \varphi(x) \\ \hline \existst x:\sigma \, \varphi(x)  \end{array}\;,
\end{align*}
where $\varphi$ ranges over internal $\lor$-free formulae.
\end{prop}
\begin{proof}
This is an adaptation of \cite[Proposition 5.12]{van2012functional}. Suppose
\begin{equation*}
	\text{H} \vdash \forallst x \, \varphi(x)\;.
\end{equation*}
By the soundness theorem, it follows that
\begin{equation*}
	\text{E-HA}^{\omega*}_\lor \vdash \forall x \, \varphi(x)\;,
\end{equation*}
which, since H is an extension of E-HA$^{\omega*}_\lor$, implies $\text{H} \vdash \forallst x \, \varphi(x)\;$.

Now, suppose
\begin{equation*}
	\text{H} \vdash \exists x \, \varphi(x)\;;
\end{equation*}
by conservativity, this implies $\text{E-HA}^{\omega*}_\lor \vdash \exists x \, \varphi(x)$. Being a subsystem of $\text{E-HA}^{\omega*}$, $\text{E-HA}^{\omega*}_\lor$ inherits the existence property; so we can find a closed term $t$ such that
\begin{equation*}
	\text{E-HA}^{\omega*}_\lor \vdash \varphi(t)\;.
\end{equation*}
Since $t$ is provably standard in H, this implies $\text{H} \vdash \existst x \, \varphi(x)\;$.
\end{proof}

\begin{prop}
The system $\text{\emph{H}} := \text{\emph{E-HA}}^{\omega*}_{\mathrm{st}\lor} + \mathsf{OS}^*_\lor + \mathsf{US}^*_\lor + \mathsf{NU} + \mathsf{AC}^\mathrm{st} + \mathsf{IP}_{\forall \, \lor}^\mathrm{st}$ has the following form of the existence property: if
\begin{equation*}
	\text{\emph{H}} \vdash \existst x \, \Phi(x)\;,
\end{equation*}
then there exists a closed term $t$ such that $\text{\emph{H}} \vdash \Phi(t)$.
\end{prop}
\begin{proof}
Let $\Phi(x)^U = \existst \tuple{u} \, \forallst \tuple{v} \, \varphi(x, \tuple{u},\tuple{v})$. By the characterisation theorem, H proves that $\Phi$ is equivalent to its $U$ translation; so, if $\text{H} \vdash \existst x \, \Phi(x)$,
\begin{equation*}
	\text{H} \vdash \existst x, \tuple{u} \, \forallst \tuple{v} \, \varphi(x, \tuple{u}, \tuple{v})\;.
\end{equation*}
By soundness of uniform Diller-Nahm, we can extract closed terms $t, \tuple{T}$ such that
\begin{equation*}
	\text{E-HA}^{\omega*}_\lor \vdash \forall \tuple{v} \, \varphi(t, \tuple{T}, \tuple{v})\;;
\end{equation*}
which, by conservativity, and weakening the quantifier, implies
\begin{equation*}
	\text{H} \vdash \forallst \tuple{v} \, \varphi(t, \tuple{T}, \tuple{v})\;.
\end{equation*}
Since the terms in $\tuple{T}$ are provably standard in H, we obtain
\begin{equation*}
	\text{H} \vdash \existst \tuple{u} \, \forallst \tuple{v} \, \varphi(t, \tuple{u}, \tuple{v})\;,
\end{equation*}
which, again by the characterisation theorem, implies $\text{H} \vdash \Phi(t)$.
\end{proof}

\begin{cor}
The system $\text{\emph{E-HA}}^{\omega*}_{\mathrm{st}\lor} + \mathsf{OS}^*_\lor + \mathsf{US}^*_\lor + \mathsf{NU} + \mathsf{AC}^\mathrm{st} + \mathsf{IP}_{\forall \, \lor}^\mathrm{st}$ has the disjunction property.
\end{cor}
\begin{proof}
Follows from the validity of $\Phi \lor \Psi \leftrightarrow \existst z:0 \, (z = 0 \to \Phi \land \neg \, z = 0 \to \Psi)$ in H, and the previous proposition.
\end{proof}

\subsection{De-herbrandisation and the topos $\mathcal{U}$}
Let us examine how de-herbrandisation is reflected in the topos-theoretic analysis. By looking at Lemma \ref{lem:quantifiers}.(b), we see that herbrandisation has a direct categorical analogue in the choice of finite families as $K$-covers. The obvious choice for uniform Diller-Nahm, then, is to consider the smaller topology $K_1$, where covers of an object $C$ are \emph{single} covering morphisms $\{f: D \twoheadrightarrow C\}$. \emph{A fortiori}, $K_1$ is also subcanonical for coherent categories.

\begin{dfn}
The Grothendieck topos $\mathrm{Sh}(\mathfrak{F}\mathbf{Set}, K_1)$ will be called $\mathcal{U}$.
\end{dfn}

The constant objects functor $\Delta_1$ of $\mathcal{U}$ is defined, for all sets $S$, at all filters $\mathcal{F}$ of $\mathfrak{F}\mathbf{Set}$, by
\begin{align*}
	(\Delta_1 S)\mathcal{F} & = \{\alpha: \mathcal{F} \to S \;|\; \alpha \text{ is constant}\}\;.
\end{align*}
We have $\Delta_1 2 \simeq 1+1 =: 2$ in $\mathcal{U}$, for $\Delta_1$ preserves coproducts; unlike in $\mathcal{N}$, there is a proper monomorphism $m: 2 \rightarrowtail \mathbf{y}2$. Moreover, since the sheafification functor $\mathbf{a}$ of $\mathcal{N}$ is itself left adjoint to the inclusion of $\mathcal{N}$ in $\mathcal{U}$, we have that $\mathbf{a}\Delta_1 2 \simeq \mathbf{y}2$. We say that $m$ is a $K$-\emph{dense} morphism.

Indeed, this fact alone characterises the topology of $\mathcal{N}$ with respect to $\mathcal{U}$. We recall a general result about elementary topoi.

\begin{prop}
Let $m: A \rightarrowtail X$ be a monomorphism in a topos $\mathcal{E}$. Then there exists a smallest local operator $j$ on $\mathcal{E}$ such that $m$ is $j$-dense.
\end{prop}
\begin{proof}
See \cite[Example A4.5.14.(b)]{johnstone2002sketches}.
\end{proof}


We can now provide a characterisation of $\mathcal{N}$ as a subtopos of $\mathcal{U}$.
\begin{prop} \label{prop:uintom}
Let $j$ be the smallest local operator on $\mathcal{U}$ such that $m: 2 \rightarrowtail \mathbf{y}2$ is $j$-dense. Then $\mathrm{sh}_j(\mathcal{U}) \simeq \mathcal{N}$.
\end{prop}
\begin{proof}
First observe that the relevant definitions imply that the $K$-covering families are precisely those finite families $\{ \mathcal{G}_k \to \mathcal{F} \}_{k=1}^n$ such that $\mathcal{G}_1 + \mathcal{G}_2 + \ldots + \mathcal{G}_k \to \mathcal{F}$ is $K_1$-covering. From this description, it follows that $K$ is the smallest topology extending $K_1$ for which also families consisting of two sum inclusions $\{ \mathcal{F}_1 \to \mathcal{F}_1 + \mathcal{F}_2, \mathcal{F}_2 \to \mathcal{F}_1 + \mathcal{F}_2 \}$ are covering.

So to prove the proposition, it suffices to show that for any local operator for which $m: 2 \rightarrowtail \mathbf{y}2$ is $j$-dense, we must have that families consisting of two sums inclusions $\{ \mathcal{F}_1 \to \mathcal{F}_1 + \mathcal{F}_2, \mathcal{F}_2 \to \mathcal{F}_1 + \mathcal{F}_2 \}$ are $j$-covering. To show this, consider the pullback square:
\begin{equation*}
\begin{tikzcd}
	\mathbf{y}\mathcal{F}_1 + \mathbf{y} \mathcal{F}_2 \arrow{r} \arrow[tail]{d} & 2 \arrow[tail]{d}{m} \\
	\mathbf{y}(\mathcal{F}_1 + \mathcal{F}_2) \arrow{r} & \mathbf{y}2,
\end{tikzcd}
\end{equation*}
where the map on the bottom of the square is obtained by applying the Yoneda embedding to the coproduct of the maps ${\mathcal F}_1 \to 1$ and ${\mathcal F}_2 \to 1$.
So if $m$ is $j$-dense, then so must be the map on the left of the square. But then it follows from Lemma \ref{lem:frommaclaneandmoerdijk} that $\{ \mathcal{F}_1 \to \mathcal{F}_1 + \mathcal{F}_2, \mathcal{F}_2 \to \mathcal{F}_1 + \mathcal{F}_2 \}$ is $j$-covering.
\end{proof}

Let $\mathcal{L}_\mathrm{st}$ be a first order language enriched with a standardness predicate, as in Section \ref{sec:filtertopos}; we interpret $\mathcal{L}_\mathrm{st}$ in $\mathcal{U}$ just as we did in $\mathcal{N}$, except that we take
\begin{enumerate}
	\item[(v')] for each type $S$, $\llbracket \mathrm{st}_S \rrbracket := \Delta_1 S$,
\end{enumerate}
so that $\llbracket \mathrm{st}_\mathbb{N} \rrbracket$ is again the natural numbers object.

As we foretold, ``internal'' becomes ``internal and $\lor$-free'' in this larger topos. Let $\Vdash_1$ be the forcing relation in $\mathcal{U}$.

\begin{thm}
Let $\varphi(x)$ be an internal, $\lor$-free formula, with free variable $x$ of type $S$, and $(C, \mathcal{F}_I)$ a filter. For all $\alpha \in \llbracket S \rrbracket \mathcal{F}$,
\begin{equation*}
	\mathcal{F} \Vdash_1 \varphi(\alpha)
\end{equation*}
if and only if there exists $i \in I$ such that, for all $u \in \mathcal{F}_i$, it holds that $\varphi(\alpha(u))$.
\end{thm}

\begin{cor}[Transfer theorem]
Let $\varphi$ be an internal and $\lor$-free sentence. Then $\varphi$ is true if and only if $\Vdash_1 \varphi$.
\end{cor}

The standardness predicate, and consequently the existential quantifier, are de-herbrandised, as we wished.
\begin{lem}
Let $\mathcal{F}$ be a filter, $S$ a type of $\mathcal{L}$, and $\alpha \in \llbracket S \rrbracket \mathcal{F}$. Then:
\begin{enumerate}[label=(\alph*)]
	\item $\mathcal{F} \Vdash_1 \mathrm{st}_S(\alpha)$ if and only if there exist a covering map $\beta: \mathcal{G} \twoheadrightarrow \mathcal{F}$, and an element $x \in S$, such that $\alpha\beta = x!$ in $\mathfrak{F}\mathbf{Set}$;
	\item $\mathcal{F} \Vdash_1 \forallst y:T \, \Phi(\alpha, y)$ if and only if, for all $y \in T$, $\mathcal{F} \Vdash_1 \Phi(\alpha, y!)$;
	\item $\mathcal{F} \Vdash_1 \existst y:T \, \Phi(\alpha, y)$ if and only if there exists $y \in T$ such that $\mathcal{F} \Vdash_1 \Phi(\alpha, y!)$.
\end{enumerate}
\end{lem}
\begin{proof}
Easy variation on the proof of Lemma \ref{lem:quantifiers}.
\end{proof}

\begin{prop}
The following principles all hold in $\mathcal{U}$: $\mathsf{OS}^*_\lor, \mathsf{US}^*_\lor, \mathsf{IP}_{\forall \, \lor}^\mathrm{st}, \mathsf{NU}$, and, if the axiom of choice holds in the metatheory, $\mathsf{AC}^\mathrm{st}$.
\end{prop}
\begin{proof}
Same as for the corresponding principles in $\mathcal{N}$, with only minor adjustments required (pick single covering maps instead of finite families).
\end{proof}

Before moving on to the conclusions, we want to remark that, irrespective of any interest in nonstandard arithmetic, and with the \emph{caveat} about $\mathsf{AC}^\mathrm{st}$ and the axiom of choice, $\mathcal{U}$ also provides a model for the logic of the ``standard'' Diller-Nahm translation, under the interpretation $\llbracket 0 \rrbracket_\land := \Delta_1 \mathbb{N}$, $\llbracket 0 \to 0 \rrbracket_\land := \Delta_1 (\mathbb{N} \to \mathbb{N})$, etc. In this case, we obtain a weaker transfer theorem for $\lor$-free formulae whose quantifiers are all bounded, i.e.\ they range over some finite sequence.

\section{Conclusions and directions for future work} \label{sec:conclusions}
If we are to sum up what we believe are \emph{conceptually} interesting points of this paper, it may come down to the following.

\begin{enumerate}
	\item \emph{Sequence overspill and underspill}. These are two principles, one dual to each other, that not only seem to be useful, and constructively acceptable generalisations of overspill and underspill to higher types; but they are also linked to well-known nonconstructive principles, the lesser limited principle of omniscience, and Markov's principle, suggesting that classical modes of reasoning can be recovered in a constructive nonstandard calculus.
	\item \emph{The significance of $\mathcal{U}$}. The filter topos $\mathcal{U}$ doubles as a model of the logic of the Diller-Nahm interpretation, and as a cue to its extension with uniform quantifiers - the uniform Diller-Nahm interpretation. This was previously unknown, and might lead to an improved understanding of the underlying, \emph{geometric} structure of Diller-Nahm logic.
	\item \emph{A better view on herbrandisation}. The comparison of $\mathcal{N}$ with $\mathcal{U}$ provided a categorical counterpart to herbrandisation, and allowed for a refined analysis of its effects. This includes the re-contextualisation of $\mathsf{NCR}$ as a herbrandised uniformity principle; on the contrary, the role of finite sequences in $\mathsf{HGMP}^\mathrm{st}$ appeared to be a byproduct of their role in $\mathsf{US}^*$, rather than the result of herbrandisation.
\end{enumerate}

A final consideration: the introduction of the nonstandard Dialectica interpretation in \cite{van2012functional} had nonstandard analysis as its main motivation; the benchmark to meet was eliminating overspill and underspill from proofs, retrieving what computational content they may have. But our analysis of uniform Diller-Nahm suggests an alternative route: one where we start with the Dialectica interpretation, and progressively apply small ``patches'', fixing whichever shortcomings might arise.

So Dialectica requires decidability of atomic formulae: we may be unhappy with that, and turn to Diller-Nahm. Then we could add uniform quantifiers \emph{\`a la} Hernest, with optimisation of program extraction in mind; which would lead us to uniform Diller-Nahm. Then, we notice that the system we obtain - $\text{E-HA}^{\omega*}_{\mathrm{st}\lor}$ and characteristic principles - is just one connective away from being a system of intuitionistic arithmetic. Also, fixing that may require that we weaken the existence property; this way, we may come up with herbrandisation, and obtain, in principle, the $D_\mathrm{st}$ interpretation, \emph{without ever actually thinking} of nonstandard arithmetic.

In fact, we can take the ``equation''
\begin{equation*}
	\text{standardness} \simeq \text{herbrandised calculability}
\end{equation*}
as a definition of sorts; one that replaces the intuition of nonstandard natural numbers as having a separate existence, lying, somewhere beyond reach, on a line together with the finite ones, with an ``operational'' interpretation: a nonstandard number is badly incalculable - so badly, that it cannot even be narrowed down to a finite selection of candidates.

We conclude with a review of new questions that our results raise.

Concerning the proof theory of nonstandard arithmetic, we would like to know how independent the principles $\mathsf{OS}^*$ and $\mathsf{US}^*$ are. We know that the Herbrand realisability interpretation vacuously accepts the former, yet does \emph{not} have a realiser for the latter \cite{van2012functional}; so the Herbrand topos from \cite{van2011herbrand} provides a model of nonstandard arithmetic with full transfer, but no underspill principle. We do not know, however, of nonstandard models where overspill holds, and underspill does not.

Moreover, we defined a new functional interpretation, but ignore, so far, how useful it is for applications. Its similarity to light Dialectica is encouraging; on the other hand, the use of functional interpretations has been most successful in program extraction from \emph{classical} proofs, and we have not investigated yet how well uniform Diller-Nahm composes with negative translations, such as Kuroda's \cite{kuroda1951intuitionistische}.
	
In light of the results of Chapter 3, Palmgren's work on the topos $\mathcal{N}$ indicates that the characteristic principles of nonstandard Dialectica lead to a useful calculus for nonstandard analysis. We conjectured that the characteristic principles of uniform Diller-Nahm may be a good axiomatisation of Lifschitz's calculability arithmetic \cite{lifschitz1985calculable}; is this correct, and could this also be a useful calculus by itself?
	
On a more speculative note, Oliva provided in \cite{oliva2006unifying} a unified view of the Dialectica, Diller-Nahm, and modified realisability interpretations, through \emph{linear logic}. Is there an equivalent of herbrandisation in linear logic - connected, perhaps, to the bang ($!$) modaliser - such that nonstandard Dialectica and Herbrand realisability, too, would be amenable to such a treatment?

We hope that these, and related questions can be answered in future work.

\bibliographystyle{plain}
\small \bibliography{bibliography}

\end{document}